\documentclass[12pt]{amsart}
\usepackage{amssymb,amsmath,amscd}
\usepackage{graphicx}
\usepackage{mathrsfs}

\newtheorem{thm}{\bf Theorem}[section]
\newtheorem{prop}[thm]{\sc Proposition}
\newtheorem{lem}[thm]{\sc Lemma}

\theoremstyle{definition}\newtheorem{exa}[thm]{\sc Example}
\theoremstyle{definition}\newtheorem{de}[thm]{\sc Definition}
\theoremstyle{definition}\newtheorem{rem}[thm]{\sc Remark}
\theoremstyle{definition}
\theoremstyle{definition}
\theoremstyle{definition}

\numberwithin{equation}{section}
\numberwithin{figure}{section}

\DeclareMathOperator{\R}{\mathbb R}

\begin{document}

\title[sample]{A characterization of constant $p$-mean curvature surfaces in the Heisenberg group $H_1$}
\author{Hung-Lin Chiu and Hsiao-Fan Liu}
\address{Department of Mathematics, National Tsing Hua University, Hsinchu, Taiwan 300, R.O.C.}
\email{hlchiu@math.nthu.edu.tw}
\address{Department of Mathematics, TamKang University,  New Taipei City 25137, Taiwan, R.O.C.}
\email{hfliu@mail.tku.edu.tw}

\begin{abstract}
In Euclidean $3$-space, it is well known that the Sine-Gordon equation was considered in the nineteenth century in the course of investigations of surfaces of constant Gaussian curvature $K=-1$. Such a surface can be constructed from a solution to the Sine-Gordon equation, and vice versa. With this as motivation, employing the fundamental theorem of surfaces in the Heisenberg group $H_{1}$, we show in this paper that the existence of a constant $p$-mean curvature surface (without singular points) is equivalent to the existence of a solution to a nonlinear second-order ODE \eqref{Codeq01}, which is a kind of {\bf Li\'{e}nard equations}.
Therefore, we turn to investigate this equation. It is a surprise that we give a complete set of solutions to \eqref{Codeq01} (or \eqref{lieq01}), and hence use the types of the solution to divide constant $p$-mean curvature surfaces into several classes. As a result, after a kind of normalization, we obtain a representation of constant $p$-mean curvature surfaces and classify further all constant $p$-mean curvature surfaces. In Section \ref{appcon}, we provide an approach to construct $p$-minimal surfaces. It turns out that, in some sense, generic $p$-minimal surfaces can be constructed via this approach. Finally, as a derivation, we recover the Bernstein-type theorem which was first shown in \cite{CHMY} (or see \cite{DGNP, DGNP1}).
\end{abstract}

\maketitle
\tableofcontents
\section{Introduction and main results}
In literature, the Heisenberg group and its sub-Laplacian are active in many fields of analysis and sub-Riemannian geometry, control theory, semiclassical analysis of quantum mechanics, etc.
(cf. \cite{C02,CCG04,H90,H84,LS95}). It also has applications in signal analysis and geometric optics \cite{F92,G00,BP08}. Research on the sub-Riemannian geometry and its analytical consequences, in particular geodesics, has been studied widely and extensively in the past ten years (cf. \cite{CCG04,BGG00, BGG96, CCG05, CM03, G77, K82,K85,KR85}). In this paper, the Heisenberg group is studied as a pseudo-hermitian manifold. Like Euclidean geometry, it is a branch of Klein geometries, and the corresponding Cartan geometry is pseudo-hermitian geometry. 

Recall that the Heisenberg group $H_1$ is the space $\mathbb{R}^3$ associated with the group multiplication
$$(x_1,y_1,z_1)\circ (x_2,y_2,z_2)
=(x_1+x_2,y_1+y_2, z_1+z_2+y_1x_2-x_1y_2),$$
which is a $3$-dimensional Lie group.  The space of all
left invariant vector fields is spanned by the following three vector fields:
$$\mathring{e}_1=\frac{\partial}{\partial x}+y\frac{\partial}{\partial z},~~
\mathring{e}_2=\frac{\partial}{\partial y}-x\frac{\partial}{\partial z}
~~\mbox{ and }~~T=\frac{\partial}{\partial z}.$$
The standard contact bundle on $H_1$ is the subbundle $\xi$ of the tangent bundle
$TH_1$ which is spanned by $\mathring{e}_1$ and $\mathring{e}_2$. It is also defined to be the kernel of
the contact form
$$\Theta=dz+xdy-ydx.$$
The CR structure on $H_1$ is the endomorphism $J:\xi\to \xi$ defined by
$$J(\mathring{e}_1)=\mathring{e}_2~~\mbox{ and }~~J(\mathring{e}_2)=-\mathring{e}_1.$$
One can view $H_1$ as a pseudo-hermitian manifold
with $(J,\Theta)$ as the standard pseudo-hermitian structure. There is a natural associated connection $\nabla$ if we regard all these left invariant vector fields $\mathring{e}_1,\mathring{e}_2$ and $T$ as parallel vector fields.
A naturally associated metric on $H_{1}$ is the adapted metric $g_{\Theta}$, which is defined by $g_{\Theta}=d\Theta(\cdot,J\cdot)+\Theta^{2}$. It is equivalent to define the metric by regarding $\mathring{e}_1,\mathring{e}_2$ and $T$ as an orthonormal frame field. We sometimes use $<\cdot,\cdot>$ to denote the adapted metric. In this paper, we use the adapted metric to measure the lengths and angles of vectors, and so on.

A pseudo-hermitian transformation (or a Heisenberg rigid motion) in $H_1$ is a diffeomorphism in $H_1$ which preserves the
standard pseudo-hermitian structure $(J,\Theta)$. We let $PSH(1)$ be the group of Heisenberg rigid motions, that is, the group of
all pseudo-hermitian transformations in $H_1$. For details of this group, we refer readers to \cite{CHMY12,CL}, in which the fundamental theorem in the Heisenberg groups has been studied. We say that two surfaces are congruent if they differ by an action of a Heisenberg rigid motion.

The concept of minimal surfaces or constant mean curvature surfaces plays an important role in differential geometry to study the basic properties of manifolds. There is an analogous concept in pseudo-hermitian manifolds, which are called $p$-minimal surfaces. In this paper, we focus on studying such kinds of surfaces in the Heisenberg group $H_{1}$. Throughout this article, all objects we discuss are assumed to be $C^{\infty}$ smooth, unless we specify otherwise. Suppose $\Sigma$ is a surface in the Heisenberg group $H_{1}$. There is a one-form $I$ on $\Sigma$ which is induced from the adapted metric $g_{\Theta}$. This induced metric is defined on the whole surface $\Sigma$ and is called the first fundamental form of $\Sigma$. The intersection $T\Sigma\cap\xi$ is integrated to be a singular foliation on $\Sigma$ called the characteristic foliation. Each leaf is called a characteristic curve. A point $p\in\Sigma$ is called a singular point if at which the tangent plane $T_{p}\Sigma$ coincides with the contact plane $\xi_{p}$; otherwise, $p$ is called a regular (or non-singular) point. Generically, a point $p\in\Sigma$ is a regular point, and the set of all regular points is called the regular part of $\Sigma$. On the regular part, we can choose a unit vector field $e_{1}$ such that $e_{1}$ defines the characteristic foliation. The vector $e_{1}$ is determined up to a sign. Let $e_{2}=Je_{1}$. Then $\{e_{1},e_{2}\}$ forms an orthonormal frame field of the contact bundle $\xi$. We usually call the vector field $e_{2}$ a horizontal vector field. Then the $p$-mean curvature $H$ of the surface $\Sigma$ is defined by
\begin{equation}
\nabla_{e_{1}}e_{2}=-He_{1}.
\end{equation}
The $p$-mean curvature $H$ is only defined on the regular part of $\Sigma$. There are two more equivalent ways to define the $p$-mean curvature from the point of view of variation and a level surface (see \cite{CHMY,M}). We remark that this notion of mean curvature was proposed by J.-H. Cheng, J.-F. Hwang, A. Malchiodi, and P. Yang from the geometric point of view to generalize the one introduced first by S. Pauls in $H_{1}$ for graphs over the $xy$-plane \cite{P}. Also, in \cite{RR}, M. Ritor\'{e} and C. Rosales exposed another method to compute the mean curvature of a hypersurface. If $H=0$ on the whole regular part, we call the surface a $p$-minimal surface. The $p$-mean curvature is the line curvature of a characteristic curve, and hence the characteristic curves are straight lines (for the detail, see \cite{CHMY}). There also exists a function $\alpha$ defined on the regular part such that $\alpha e_{2}+T$ is tangent to the surface $\Sigma$. We call this function the $\alpha$-function of $\Sigma$. It is uniquely determined up to a sign, which depends on the choice of the characteristic direction $e_{1}$. Define $\hat{e}_{1}=e_{1}$ and $\hat{e}_{2}=\frac{\alpha e_{2}+T}{\sqrt{1+\alpha^2}}$, then $\{\hat{e}_{1},\hat{e}_{2}\}$ forms an orthonormal frame field of the tangent bundle $T\Sigma$. Notice that $\hat{e}_{2}$ is uniquely determined and independent of the choice of the characteristic direction $e_{1}$. In \cite{CL}, it was shown that these four invariants, 
\[I,e_{1},\alpha,H,\]
form a complete set of invariants for surfaces in $H_{1}$. We remark that all the results provided in \cite{CL} still hold in the $C^2$-category. For each regular point $p$, we can choose suitable coordinates around $p$ to study the local geometry of such surfaces. There always exists a coordinate system $(x,y)$ of $p$ such that 
\[e_{1}=\frac{\partial}{\partial x}.\]
We call such coordinates $(x,y)$ a {\bf compatible coordinate system}. It is determined up to a transformation in \eqref{cortra}. Notice that the compatible coordinate systems are dependent on the characteristic direction $e_{1}$.

Let $\Sigma\subset H_{1}$ be a constant $p$-mean curvature surface with $H=c$. In terms of a compatible coordinate system $(U;x,y)$, the $\alpha$-function satisfies the following equation
\begin{equation}\label{Codeq01}
\alpha_{xx}+6\alpha\alpha_{x}+4\alpha^{3}+c^{2}\alpha=0,
\end{equation}
which first appeared as a Codazzi-like equation $(1.12)$ in \cite{CHMY12} with $D= -1/{\alpha}$.

It is a nonlinear ordinary differential equation and is an example of the so-called {\bf Li\'{e}nard equations}, named after the French physicist Alfred-Marie Li\'{e}nard. The Li\'{e}nard equations were intensely studied as they can be used to model oscillating circuits. Conversely, in this paper, we show that if there exists a solution $\alpha(x,y)$ to the {\bf Li\'{e}nard equation} \eqref{Codeq01}, we are able to construct a constant $p$-mean curvature surface with $H=c$ and this given solution $\alpha$ as its $\alpha$-function. This result is summarized as Theorem \ref{main021}. One motivation of this theorem comes from the famous Sine-Gordon Equation (SGE), 
\[ u_{tt}-u_{xx} = \sin(u)\cos(u),\] which is considerably older than the Korteweg de Vries Equation (KdV). It was discovered in the late eighteenth century in the study of {\it pseudospherical} surfaces, that is, surfaces of Gaussian curvature $K=-1$ immersed in $\R^3$, and it was intensively studied for this reason. It arises from the Gauss-Codazzi equations for pseudospherical surfaces in $\R^3$ and is known as an integrable equation \cite{AKNS}. In addition, it can also be viewed as a continuum limit \cite{RM94}. Its solutions and solitons have been widely discussed by the Inverse Scattering Transform and other approaches. 

There is a bijective relation between solutions $u$ of the SGE with $\Im(u) \subset (0,\frac{\pi}{2})$ and the classes of pseudospherical surfaces in $\R^3$ up to rigid motion. If $u : \R^2 \rightarrow \R$ is a solution such that $\sin u\cos u$ is zero at a point $u_0$, then the immersed pseudoshperical surface has cusp singularities. For example, the pseudospherical surfaces corresponding to the 1-soliton solutions of SGE are the so-called Dini surfaces and have a helix of singularities.

The study of line congruences gives rise to the concept of B\"acklund transformations. A line congruence $L: M \rightarrow M^*$ is called a B\"acklund transformation with the constant angle $\theta$ between the normal to $M$ at $p$ and the normal to $M^*$ at $p^*= L(p)$ and the distance between $p$ and $p^*$ is $\sin\theta$ for all $p \in M$. The classical B\"acklund transformation for the Sine-Gordon equation was constructed in the nineteenth century by Swedish differential geometer Albert B\"acklund by means of a geometric construction \cite{B82,B83,B13}. We then are motivated to present analogous theorems for Heisenberg groups.

\begin{thm}\label{main021}
The existence of a constant $p$-mean curvature surface $($without singular points$)$ in $H_{1}$ is equivalent to the existence of a solution to the equation \eqref{Codeq01}. More precisely, let $\alpha(x,y)$ and $H(x,y)$ be two arbitrary smooth functions on a coordinate neighborhood $(U;x,y)\subset\R^{2}$.  If they satisfy the following integrability condition
\begin{equation}\label{intcon01}
\left(h(y)-\int H\alpha \left(e^{\int 2\alpha dx}\right) dx\right)H_{x}+e^{k(y)}H_{y}=\left(e^{\int 2\alpha dx}\right)(\alpha_{xx}+6\alpha\alpha_{x}+4\alpha^{3}+\alpha H^{2}),
\end{equation}
for some functions $k(y)$ and $h(y)$ in the variable $y$, then there exists an embedding $X: U\rightarrow H_{1}$ $($provided that $U$ is small enough$)$ such that the surface $\Sigma=X(U)$ has $H$ and $\alpha$ as its $p$-mean curvature and $\alpha$-function, respectively, and $\hat{e}_{1}=\frac{\partial}{\partial x}$,\ $\hat{e}_{2}=a(x,y)\frac{\partial}{\partial x}+b(x,y)\frac{\partial}{\partial y}$ with $a$ and $b$ defined as \eqref{foforb} and \eqref{fofora}. In addition, such embeddings are unique, up to a Heisenberg rigid motion.

In particular, when $H=c$ is a constant, the integrability condition \eqref{intcon01} reads \eqref{Codeq01} for each pair of functions $k(y)$ and $h(y)$, and the resulting surface $\Sigma=X(U)$ is a {\bf constant} $p$-mean curvature surface with $H=c$, the given function $\alpha(x,y)$ as its $\alpha$-function, and $\hat{e}_{1}=\frac{\partial}{\partial x}$.
\end{thm}

In this article, we sometimes call the {\bf Li\'{e}nard equation} \eqref{Codeq01} as the {\bf Codazzi-like} equation from the geometrical point of view \cite{CHMY12,CL}. We would also like to specify that a graph $(x,y,u(x,y))$ is $p$-minimal if it satisfies the $p$-minimal equation (see \cite{CHMY})
\begin{equation}\label{pmge}
(u_{y}+x)^{2}u_{xx}-2(u_{y}+x)(u_{x}-y)u_{xy}+(u_{x}-y)^{2}u_{yy}=0.
\end{equation} 
This is a degenerate hyperbolic and elliptic partial differential equation. 

Theorem \ref{main021} is the fundamental theorem for surfaces in $H_{1}$. After we make a detailed investigation of the original version of the integrability conditions \eqref{intcon}, \eqref{intcon01} is more useful than the previous one in some sense (for the origin version, we refer readers to \cite{CL} or Theorem \ref{fudthm} of this paper). It also appeared as Theorem H in \cite{CHMY12} with a slightly different formulation as the authors of \cite{CHMY12} did not prescribe the metric.

Theorem \ref{main021} follows from our detailed study on the integrability condition (see \eqref{intcon}) of the fundamental theorem (Theorem \ref{fudthm}) for surfaces in $H_{1}$. Actually, if we let $\hat{\omega}_{1}{}^{2}$ be the Levi-Civita connection associated to the induced metric with respect to the orthonormal frame field $\{\hat{e}_{1},\hat{e}_{2}\}$, as specified in Theorem \ref{main021}, 
then \eqref{intcon01} means that $\hat{\omega}_{1}{}^{2}, \alpha$, and $H$ satisfy the integrability condition \eqref{intcon}. This is equivalent to saying that $a, b, \alpha$ and $H$ satisfy the integrability condition \eqref{intcon1} (see Subsection \ref{neveinco}), which is another version of \eqref{intcon}. We then have Theorem \ref{main021}.

Given a function $\alpha(x,y)$ in a coordinate neighborhood $(U;x,y)\subset\R^2$ which satisfies the {\bf Codazzi-like} equation \eqref{Codeq01}, we are able to construct a family of constant $p$-mean curvature surfaces. Therefore, it suggests that a good strategy is to investigate constant $p$-mean curvature surfaces by means of the {\bf Codazzi-like} equation \eqref{Codeq01}; in particular, $p$-minimal surfaces. In this paper, we will focus on the theory of $p$-minimal surfaces. Strategically, we first study the equation \eqref{Codeq01} with $c=0$, that is,
\begin{equation}\label{lieq01}
\alpha_{xx}+6\alpha\alpha_{x}+4\alpha^{3}=0.
\end{equation}
For nonlinear ordinary differential equations, it is known that it is rarely possible to find explicit solutions in closed form, even in power series. Fortunately, we indeed obtain a complete set of solutions to \eqref{lieq01} in a simple form (see Section \ref{sofsotl}), stated as follows.

\begin{thm}\label{main031}
Besides the following three special solutions to \eqref{lieq01},  
\begin{equation*}
\alpha(x)=0,\ \frac{1}{x+c_{1}},\ \frac{1}{2(x+c_{1})},
\end{equation*}
we have the general solution to \eqref{lieq01} of the form
\begin{equation}
\alpha(x)=\frac{x+c_{1}}{(x+c_{1})^{2}+c_{2}},
\end{equation}
which depends on two constants $c_{1}$ and $c_{2}$, and $c_{2}\neq 0$.
\end{thm}
In Subsection \ref{claofpmi}, we are able to use the types of the solutions in Theorem \ref{main031} to divide the $p$-minimal surfaces into several classes, which are {\bf vertical}, {\bf special type I}, {\bf special type II} and {\bf general type} (see Definition \ref{cladef} and \ref{cladef1}). Each type of these $p$-minimal surfaces is open and contains no singular points. Generically, each $p$-minimal surface is a union of these types of surfaces, and "type" can be shown to be invariant under an action of a Heisenberg rigid motion. Now for each type, whether it is special or general, if a function $\alpha$ is given, 
then Proposition \ref{indmetin1}, \ref{indmetin2} and \ref{indmetin3} express the formula for the induced metric $a,b$ (see \eqref{foab42}, \eqref{foab422} and \eqref{foab423}), which is a representation of $I$, on the $p$-minimal surfaces with this given $\alpha$ as $\alpha$-function. From these formulae, we see that such constructed $p$-minimal surfaces depend upon two functions $k(y)$ and $h(y)$ for each given $\alpha$. 
 Nonetheless, in Section \ref{othocs}, we proceed to normalize these invariants to the following normal forms in terms of an orthogonal coordinate system $(x,y)$, which is a coordinate system such that $a=0$. Such a coordinate system is determined up to a translation on $(x,y)$, thus we call it a {\bf normal coordinate system}. 
 \begin{thm}\label{main041}
 Let $\Sigma\subset H_{1}$ be a $p$-minimal surface. Then, in terms of a normal coordinate system $(x,y)$, we can normalize the $\alpha$-function and the induced metric $a,b$ to be the following {\bf normal forms}: 
\begin{enumerate}
\item $\alpha=\frac{1}{x+\zeta_{1}(y)}$, and $a=0, b=\frac{\alpha^2}{\sqrt{1+\alpha^{2}}}$ if $\Sigma$ is of {\bf special type I},
\item $\alpha=\frac{1}{2x+\zeta_{1}(y)}$, and $a=0, b=\frac{|\alpha|}{\sqrt{1+\alpha^{2}}}$ if $\Sigma$ is of {\bf special type II},
\item $\alpha=\frac{x+\zeta_{1}(y)}{(x+\zeta_{1}(y))^{2}+\zeta_{2}(y)}$, and $a=0, b=\frac{|\alpha|}{|x+\zeta_{1}(y)|\sqrt{1+\alpha^{2}}}$ if $\Sigma$ is of {\bf general type},
\end{enumerate} 
for some functions $\zeta_{1}(y)$ and $\zeta_{2}(y)$, in which $\zeta_{2}(y)$ is unique up to a translation on $y$, and 
$\zeta_{1}(y)$ is unique up to a translation on $y$ as well as its image.
 \end{thm}
 
 Therefore, $\zeta_{1}(y)$ constitutes a complete set of invariants for $p$-minimal surfaces of special type I (or of special type II), whereas both $\zeta_{1}(y)$ and 
$\zeta_{2}(y)$ constitute a complete set of invariants for $p$-minimal surfaces of general type. We hence give a {\bf representation} for $p$-minimal surfaces (see Section \ref{othocs}). From Theorem \ref{main041}, together with Theorem \ref{main021}, the following version of the fundamental theorem holds immediately for $p$-minimal surfaces in $H_{1}$.

Analogs for constant $p$-mean curvature surfaces are also
 included in Subsection \ref{norconstp}. We first derive the general form of $\alpha$ to the Codazzi-like equation \eqref{Codeq01} for $c\neq 0$ ( stated as Theorem \ref{consta}), and Proposition \ref{metricform} provides the formula for the induced metric $a$ and $b$. 
\begin{thm}\label{consta}
Besides the trivial solution $\alpha(x)=0$, the Codazzi-like equation \eqref{Codeq01} for nonzero $c$  
has the special solutions 
\begin{equation*}
-\frac{1}{2}|c|\tan(|c|x+|c|K_1),\ \alpha(x)=\frac{1}{2}|c|\tan(-\frac{|c|}{2}x+\frac{|c|}{2}K_2),\quad K_1,K_2\in \R,
\end{equation*}
and the general solution of the form
\begin{equation}
\alpha(x)=\lambda\frac{\sin{(2\lambda x+c_{1}})}{c_{2}-\cos{(2\lambda x+c_{1})}},
\end{equation}
which depends on two constants $c_1$ and $c_2$, and $c=2\lambda$. 
\end{thm}
Since each function of $\sin(x)$ and $\cos(x)$ differs by a sign if we replace the angle $x$ by $x+\frac{\pi}{2\lambda}$, we can assume, without loss of generality, that $c_{2}\geq 0$. We use $c_{2}$ to divide constant $p$-mean curvature surfaces into two classes. One class is those surfaces with $c_{2}>1$, and the other is that with $0\leq c_{2}\leq 1$. It is worth our attention that, for surfaces with $c_{2}>1$, the denominator of the formula for $\alpha$ is never zero. That means that the surfaces won't extend to one with singular points. Moreover, if the surface is closed, it will be a closed constant $p$-mean curvature surface without singularity, which means this surface is of type of torus. It indicates that it is possible to find a Wente-type torus in the class of these surfaces. We can also normalize these invariants to the normal forms. However, Proposition \ref{metricform} implies that the normalization for constant $p$-mean curvature surfaces needs to be modified since $a$ must be nonzero. Therefore, the transformation laws help us obtain the normal coordinates $(\tilde x, \tilde y)$ in Subsection \ref{norconstp1} and hence we have the complete set of invariants $\{\zeta_1(y),\zeta_2(y)\}$ for constant $p$-mean curvature surfaces (see Theorem \ref{cominv}).

\begin{thm}\label{main051}
Given two arbitrary functions $\zeta_{1}(y)$ and $\zeta_{2}(y)$ defined on $(c,d)\subset\R$, and $\zeta_{2}(y)\neq 0$ for all $y\in(c,d)$ (note that $(c,d)$ may be the whole line $\R$), then 
\begin{enumerate}
\item there exist an open set $U\subset(e,f)\times(c,d)\subset(\R^2;x,y)$, for some $(e,f)\subset\R$, and an embedding $X:U\rightarrow H_{1}$ such that $\Sigma=X(U)$ is a $p$-minimal surface of {\bf special type I} $($or of {\bf special type II}$)$ with $(x,y)$ as a normal coordinate system and $\zeta_{1}(y)$ as its $\zeta_{1}$-invariant; 
\item there exist an open set $U\subset(e,f)\times(c,d)\subset(\R^2;x,y)$, for some $(e,f)\subset\R$, and an embedding $X:U\rightarrow H_{1}$ such that $\Sigma=X(U)$ is a $p$-minimal surface of {\bf general type} with $(x,y)$ as a normal coordinate system and $\zeta_{1}(y)$ and $\zeta_{2}(y)$ as its $\zeta_{1}$- and $\zeta_{2}$-invariants. 
\end{enumerate}
Moreover, such embeddings in $(1)$ and $(2)$ are unique, up to a Heisenberg rigid motion.
\end{thm}

Due to Theorem \ref{main051}, for each pair of function $\zeta_{1}(y)$ and $\zeta_{2}(y)$, we define in Subsection \ref{mami} eight maximal $p$-minimal surfaces in the sense specified in Theorem \ref{main061}. Roughly speaking, it says that any connected $p$-minimal surface with type is a part of one of these eight classes. One notices that generically, a $p$-minimal surface is a union of those $p$-minimal surfaces with type.

\begin{thm}\label{main061}
Given two arbitrary functions $\zeta_{1}(y)$ and $\zeta_{2}(y)$ defined on $(c,d)\subset\R$, and $\zeta_{2}(y)\neq 0$ for all $y\in(c,d)$ $($note that $(c,d)$ may be the whole line $\R$$)$, then all the eight $p$-minimal surfaces 
\begin{equation}
\begin{split}
&S_{I}^{-}(\zeta_{1}),\ S_{I}^{+}(\zeta_{1}),\ S_{II}^{-}(\zeta_{1}),\ S_{II}^{+}(\zeta_{1});\ \textrm{and}\\
&\Sigma_{I}(\zeta_{1},\zeta_{2}),\ \Sigma_{II}^{-}(\zeta_{1},\zeta_{2}),\ \Sigma_{II}^{+}(\zeta_{1},\zeta_{2})\ \textrm{and}\ \Sigma_{III}(\zeta_{1},\zeta_{2})
\end{split}
\end{equation}
are immersed, in addition, they are {\bf maximal} in the following sense:
\begin{itemize} 
\item Any connected $p$-minimal surface of {\bf special type I} with $\zeta_{1}(y)$ as the $\zeta_{1}$-invariant is a part of either $S_{I}^{-}(\zeta_{1})$ or $S_{I}^{+}(\zeta_{1})$.
\item Any connected $p$-minimal surface of {\bf special type II} with $\zeta_{1}(y)$ as the $\zeta_{1}$-invariant is a part of either $S_{II}^{-}(\zeta_{1})$ or $S_{II}^{+}(\zeta_{1})$.
\item Any connected $p$-minimal surface of {\bf type I} with $\zeta_{1}(y)$ and $\zeta_{2}(y)$ as the $\zeta_{1}$- and $\zeta_{2}$-invariants is a part of  $\Sigma_{I}(\zeta_{1},\zeta_{2})$.
\item Any connected $p$-minimal surface of {\bf type II} with $\zeta_{1}(y)$ and $\zeta_{2}(y)$ as the $\zeta_{1}$- and $\zeta_{2}$-invariants is a part of either $\Sigma_{II}^{-}(\zeta_{1},\zeta_{2})$ or $\Sigma_{II}^{+}(\zeta_{1},\zeta_{2})$.
\item Any connected $p$-minimal surface of {\bf type III} with $\zeta_{1}(y)$ and $\zeta_{2}(y)$ as the $\zeta_{1}$- and $\zeta_{2}$-invariants is a part of  $\Sigma_{III}(\zeta_{1},\zeta_{2})$.
\end{itemize}
\end{thm}

As applications of this theory, in Section \ref{strofsing}, we give a complete description of the structures of the singular sets of $p$-minimal surfaces in the Heisenberg group $H_{1}$. 

\begin{thm}\label{strofsiset}
The singular set of a $p$-minimal surface is either 
\begin{enumerate}
\item an isolated point; or 
\item a $C^1$ smooth curve.   
\end{enumerate}
In addition, an isolated singular point only happens in the surfaces of special type I with $\zeta_{1}=\textrm{const.}$, that is, a part of the graph $u=0$ contains the origin as the isolated singular point.
\end{thm}
The result in Theorem \ref{strofsiset} is just a special case of Theorem 3.3 in \cite{CHMY}. However, we give a computable proof of this result for $p$-minimal surfaces. We also have the description of how a characteristic leaf goes through a singular curve, which is called a "go through" theorem in \cite{CHMY}.

\begin{thm}\label{theo2} 
Let $\Sigma\subset H_{1}$ be a $p$-minimal surface. Then the characteristic foliation is smooth around the singular curve in the following sense that each leaf can be extended smoothly to a point on the singular curve.
\end{thm}

Due to Theorem \ref{theo2}, we have the following result.
\begin{thm}\label{main10}
Let $\Sigma$ be a $p$-minimal surface of type II $($III$)$. If it can be smoothly extended through the singular curve, then the other side of the singular curve is of type III $($II$)$.
\end{thm}

Theorem \ref{main10} plays a key point to enable us to recover the Bernstein-type theorem (see Section \ref{strofsing}), which was first shown in the original paper \cite{CHMY} (or see \cite{ACV,DGNP,DGNP1}), and says that
\[u(x,y)=Ax+By+C,\]
for some constants $A,B,C\in\R$, and 
\[u(x,y)=-ABx^{2}+(A^{2}-B^{2})xy+ABy^{2}+g(-Bx+Ay),\]
where $A,B$ are constants such that $A^2+B^2=1$ and $g\in C^{\infty}(\R)$, are the only two classes of entire smooth solutions to the $p$-minimal graph equation \eqref{pmge}. In addition, in Section \ref{examples}, we present some basic examples which, in particular, help us figure out the Bernstein-type theorem. 

Finally, in Section \ref{appcon}, depending on a parametrized curve $\mathcal{C}(\theta)=(x(\theta),y(\theta),z(\theta))$ for $\theta\in\R$, we deform the graph $u=0$ in some way to construct $p$-minimal surfaces with parametrization
 \begin{equation}\label{paraformini}
 Y(r,\theta)=(x(\theta)+r\cos{\theta},y(\theta)+r\sin{\theta},z(\theta)+ry(\theta)\cos{\theta}-rx(\theta)\sin{\theta}),
 \end{equation}
for $r\in\R$. It is easy to checked that $Y$ is {\bf an immersion} if and only if either $\Theta(\mathcal{C}'(\theta))-(y'(\theta)\cos{\theta}-x'(\theta)\sin{\theta})^{2}\neq 0\ \textrm{or}\ r+(y'(\theta)\cos{\theta}-x'(\theta)\sin{\theta})\neq 0$ for all $\theta$ (see Remark \ref{re92}).
In particular, 
% \begin{thm}\label{main20}
the surface $Y$ defines a $p$-minimal surface of {\bf special type I} if the curve $\mathcal{C}$ satisfies
\begin{equation}\label{707}
z'(\theta)+x(\theta)y'(\theta)-y(\theta)x'(\theta)-(y'(\theta)\cos{\theta}-x'(\theta)\sin{\theta})^2=0,
\end{equation}
for all $\theta$, or equivalently
\begin{equation}\label{708}
z(\theta)=\int\big[(y'(\theta)\cos{\theta}-x'(\theta)\sin{\theta})^{2}+y(\theta)x'(\theta)-x(\theta)y'(\theta)\big] d\theta.
\end{equation}
In addition, the corresponding $\zeta_{1}$-invariant reads
\begin{equation}\label{713}
\zeta_{1}(\theta)=y'(\theta)\cos{\theta}-x'(\theta)\sin{\theta}-\int \big[x'(\theta)\cos{\theta}+y'(\theta)\sin{\theta}\big] d\theta,
\end{equation}
where $\int \big[x'(\theta)\cos{\theta}+y'(\theta)\sin{\theta}\big] d\theta$ is an anti-derivative of the function $x'(\theta)\cos{\theta}+y'(\theta)\sin{\theta}$.
%\end{thm}

Similarly, 
%\begin{thm}\label{main21}
the surface $Y$ defines a $p$-minimal surface of {\bf general type} if the curve $\mathcal{C}$ satisfies
\begin{equation}\label{7071}
z'(\theta)+x(\theta)y'(\theta)-y(\theta)x'(\theta)-(y'(\theta)\cos{\theta}-x'(\theta)\sin{\theta})^2\neq 0,
\end{equation}
for all $\theta$. In addition, the corresponding $\zeta_{1}$- and $\zeta_{2}$-invariant read
\begin{equation}\label{714}
\left\{\begin{split}
\zeta_{1}(\theta)&=y'(\theta)\cos{\theta}-x'(\theta)\sin{\theta}-\int \big[x'(\theta)\cos{\theta}+y'(\theta)\sin{\theta}\big] d\theta\\
\zeta_{2}(\theta)&=z'(\theta)+x(\theta)y'(\theta)-y(\theta)x'(\theta)-(y'(\theta)\cos{\theta}-x'(\theta)\sin{\theta})^2.
\end{split}\right.
\end{equation}
%\end{thm}

From this construction, together with Theorem \ref{main09} which gives parametrizations for $p$-minimal surfaces of special type II, we conclude that we have generically provided a parametrization  
for any given $p$-minimal surface (see the argument in Section \ref{appcon}).\\

{\bf Acknowledgments.} The first author's research was supported in part by NCTS and in part by MOST 109-2115-M-007-004 -MY3. The second author's research was supported in part by MOST 108-2115-M-032-008-MY2 and 110-2115-M-032-005-MY2.

\section{The Fundamental Theorem for surfaces in $H_{1}$}
 In this section, we first review the fundamental theorem for surfaces in the Heisenberg group $H_{1}$ (Theorem \ref{fudthm}). For the details, we refer the reader to \cite{CL}. Next, we give another version (Theorem \ref{main021}) of this theorem in terms of compatible coordinate systems.

\subsection{The Fundamental Theorem for surfaces in $H_{1}$}
Recall that there are four invariants for surfaces induced on a surface $\Sigma$ from the Heisenberg group $H_{1}$:
\begin{enumerate}
\item The first fundamental form (or the induced metric) $I$, which is the adapted metric $g_{\Theta}$ restricted to $\Sigma$. This metric is actually defined on the whole surface $\Sigma$.
\item The directed characteristic foliation $e_{1}$, which is a unit vector field $\in T\Sigma\cap\xi$. This vector field is only defined on the regular part of $\Sigma$.
\item The $\alpha$-function $\alpha$, which is a function defined on the regular part such that $\alpha e_{2}+T\in T\Sigma$, where $e_{2}=Je_{1}$.
\item The $p$-mean curvature $H$, which is a function on the regular part defined by $\nabla_{e_{1}}e_{2}=-He_{1}$.
\end{enumerate}
These four invariants constitute a complete set of invariants for surfaces in $H_{1}$. That is, if $\phi:\Sigma_{1}\rightarrow\Sigma_{2}$ is a diffeomorphism between these two surfaces which preserves these four invariants, then $\phi$ is the restriction of a Heisenberg rigid motion $\Phi$. We have the {\bf integrability condition}
\begin{equation}\label{intcon}
\begin{split}
\hat{\omega}_{1}{}^{2}(\hat{e}_{1})&=\frac{H\alpha}{(1+\alpha^2)^{1/2}},\\
\hat{\omega}_{1}{}^{2}(\hat{e}_{2})&=2\alpha+\frac{\alpha(\hat{e}_{1}\alpha)}{1+\alpha^2},\\
\hat{e}_{2}H&=\frac{\hat{e}_{1}\hat{e}_{1}\alpha+6\alpha(\hat{e}_{1}\alpha)+4\alpha^{3}+\alpha H^{2}}{(1+\alpha^2)^{1/2}},
\end{split}
\end{equation}
where $\hat{e}_{2}=\frac{\alpha e_{2}+T}{\sqrt{1+\alpha^{2}}}$, which is nothing to do with the orientation of $\Sigma$ but a vector field and only determined by the contact form $\Theta$. Moreover, $\hat{e}_{1}=e_{1}$, which is the characteristic direction and is determined up to a sign (if we choose $\hat{e}_{1}$ such that $\hat{e}_{1}\wedge\hat{e}_{2}$ is compatible with the orientation of $\Sigma$ then $\hat{e}_{1}$ is unique). The form $\hat{\omega}_{1}{}^{2}$ is the Levi-Civita connection form with respect to the frame $\{\hat{e}_{1},\hat{e}_{2}\}$. We present the following fundamental theorem (see \cite{CL}).

\begin{thm}[{\bf The Fundamental theorem for surfaces in $H_{1}$}]\label{fudthm}
Let $(\Sigma ,g)$ be a Riemannian $2$-manifold, 
and let $\hat{\alpha},\hat{H}$ be two real-valued functions on $\Sigma $. Assume that $g$, together with $\hat{\alpha},\hat{H}$, 
satisfies the integrability condition \eqref{intcon}, with $\alpha ,H$ replaced by $\hat{\alpha},\hat{H}$, respectively. Then for every point $p\in\Sigma $, there exist an open neighborhood $U$ containing $p$, and an
embedding $X:U\rightarrow H_{1}$ such that $g=X^{\ast }(I), \hat{\alpha}=X^{\ast }\alpha $ and $\hat{H}=X^{\ast }H$.
And $X_{\ast}(\hat{e}_{1})$ defines the foliation on $X(U)$ induced from $H_{1}$.
Moreover, $X$ is unique up to a Heisenberg rigid motion.
\end{thm}

\subsection{The new version of the integrability conditions}\label{neveinco}
The goal of this subsection is to express the integrability condition \eqref{intcon} in terms of a compatible coordinate system $(x,y)$. We write 
\begin{equation}
\hat{e}_{2}=a(x,y)\frac{\partial}{\partial x}+b(x,y)\frac{\partial}{\partial y},
\end{equation}
for some functions $a$ and $b\neq 0$. We can assume, without loss of generality, that $b>0$, that is, both $\frac{\partial}{\partial x}\wedge\frac{\partial}{\partial y}$ and $\hat{e}_{1}\wedge\hat{e}_{2}$ define the same orientation on $\Sigma$.
The two functions $a$ and $b$ are a representation of the first fundamental form $I$. The dual co-frame $\{\hat{\omega}^{1},\hat{\omega}^{2}\}$ of $\{\hat{e}_{1},\hat{e}_{2}\}$ is
\begin{equation}\label{indcof}
\begin{split}
\hat{\omega}^{1}&=dx-\frac{a}{b}dy,\\
\hat{\omega}^{2}&=\frac{1}{b}dy.
\end{split}
\end{equation}
Then the Levi-Civita connection forms are uniquely determined by the following Riemannian structure equations
\begin{equation}\label{riestr}
\begin{split}
d\hat{\omega}^{1}&=\hat{\omega}^{2}\wedge\hat{\omega}_{2}{}^{1},\\
d\hat{\omega}^{2}&=\hat{\omega}^{1}\wedge\hat{\omega}_{1}{}^{2},
\end{split}
\end{equation}
with the normalized condition
\begin{equation}\label{norli}
\hat{\omega}_{1}{}^{2}+\hat{\omega}_{2}{}^{1}=0.
\end{equation}
A computation shows that
\begin{equation}
\begin{split}
d\hat{\omega}^{1}&=-d\left(\frac{a}{b}\right)\wedge dy=-\left(\frac{bda-adb}{b^{2}}\right)\wedge dy\\
&=\frac{dy}{b}\wedge\frac{bda-adb}{b}=\hat{\omega}^{2}\wedge\frac{bda-adb}{b}.
\end{split}
\end{equation}
By comparing with the first equation of the Riemannian structure equations (\ref{riestr}), we have 
\begin{equation}\label{confo}
\hat{\omega}_{2}{}^{1}=\frac{bda-adb}{b}+a_{11}\hat{\omega}^{2}
\end{equation}
for some function $a_{11}$. On one hand, the second equation of (\ref{riestr}) and the normalized condition (\ref{norli}) imply
\begin{equation}\label{comofsec1}
\begin{split}
d\hat{\omega}^{2}&=\hat{\omega}^{1}\wedge\hat{\omega}_{1}{}^{2}\\
&=-(dx-\frac{a}{b}dy)\wedge\left(\frac{bda-adb}{b}+a_{11}\hat{\omega}^{2}\right)\\
&=\left(-a_{y}+\frac{a}{b}b_{y}-\frac{a_{11}}{b}-\frac{a}{b}a_{x}+\frac{a^2}{b^2}b_{x}\right)dx\wedge dy.
\end{split}
\end{equation}
On the other hand, one sees
\begin{equation}\label{comofsec2}
d\hat{\omega}^{2}=\left(d\frac{1}{b}\right)\wedge dy=-\frac{b_{x}}{b^{2}}dx\wedge dy.
\end{equation}
Equations (\ref{comofsec1}) and (\ref{comofsec2}) yield
\begin{equation}
a_{11}=\frac{b_{x}}{b}-ba_{y}+ab_{y}-aa_{x}+\frac{a^2}{b}b_{x}.
\end{equation}
Therefore,
\begin{equation}\label{confor}
\begin{split}
\hat{\omega}_{2}{}^{1}&=\frac{bda-adb}{b}+a_{11}\hat{\omega}^{2}\\
&=(ba_{x}-ab_{x})\frac{dx}{b}+(ba_{y}-ab_{y})\frac{dy}{b}+ \left(\frac{b_{x}}{b}-ba_{y}+ab_{y}-aa_{x}+\frac{a^2}{b}b_{x}\right)\frac{dy}{b}\\
&=\left(\frac{ba_{x}-ab_{x}}{b}\right)dx+\left(\frac{b_{x}}{b^2}-\frac{aa_{x}}{b}+\frac{a^{2}b_{x}}{b^2}\right)dy.
\end{split}
\end{equation}
From the connection form formula (\ref{confor}), we have
\begin{equation}
\begin{split}
\hat{\omega}_{1}{}^{2}(\hat{e}_{1})&=\frac{ab_{x}-ba_{x}}{b}=-a_{x}+a\frac{b_{x}}{b},\\
\hat{\omega}_{1}{}^{2}(\hat{e}_{2})&=\frac{a(ab_{x}-ba_{x})}{b}+b\left(\frac{aa_{x}}{b}-\frac{b_{x}}{b^2}-\frac{a^{2}b_{x}}{b^2}\right)=-\frac{b_{x}}{b}.
\end{split}
\end{equation}
Therefore, in terms of a compatible coordinate system $(U; x,y)$, the integrability condition (\ref{intcon}) is equivalent to 
\begin{equation}\label{intcon1}
\begin{split}
-a_{x}+a\frac{b_{x}}{b}&=\frac{H\alpha}{(1+\alpha^2)^{1/2}},\\
-\frac{b_{x}}{b}&=2\alpha+\frac{\alpha\alpha_{x}}{1+\alpha^2},\\
aH_{x}+bH_{y}&=\frac{\alpha_{xx}+6\alpha\alpha_{x}+4\alpha^{3}+\alpha H^{2}}{(1+\alpha^2)^{1/2}}.
\end{split}
\end{equation}

\subsection{The computation of the first fundamental form $I$}
We would like to solve the first two equations of (\ref{intcon1}), which are part of the integrability condition. From the second equation of (\ref{intcon1}), it is easy to see that 
\begin{equation}
\ln{|b|}=\int-\left(2\alpha+\frac{\alpha\alpha_{x}}{1+\alpha^2}\right)dx+k(y)
\end{equation}
for some function $k(y)$ in the variable $y$, that is
\begin{equation}\label{foforb}
\begin{split}
|b|&=e^{k(y)}e^{\int-\left(2\alpha+\frac{\alpha\alpha_{x}}{1+\alpha^2}\right)dx},\\
&=e^{k(y)}\frac{e^{-\int2\alpha dx}}{(1+\alpha^2)^{\frac{1}{2}}},
\end{split}
\end{equation}
where $\int 2\alpha dx$ is an anti-derivative of $2\alpha$ with respect to $x$. Throughout this paper, we always assume, without loss of generality, that $b>0$.
For $a$, we substitute the second equation of (\ref{intcon1}) into the first one to obtain the first-order linear ODE 
\begin{equation}
a_{x}+a\left(2\alpha+\frac{\alpha\alpha_{x}}{1+\alpha^2}\right)+\frac{H\alpha}{(1+\alpha^2)^{1/2}}=0.
\end{equation}
To solve $a$, we choose the integrating factor $u=e^{\int\left(2\alpha+\frac{\alpha\alpha_{x}}{1+\alpha^2}\right)dx}$ such that the one-form 
\begin{equation}
u\left(\left[a\left(2\alpha+\frac{\alpha\alpha_{x}}{1+\alpha^2}\right)+\frac{H\alpha}{(1+\alpha^2)^{1/2}}\right]dx+da\right)
\end{equation}
is an exact form. Therefore, using the standard method of ODE, one sees that
\begin{equation}\label{fofora}
\begin{split}
a&=e^{\int-\left(2\alpha+\frac{\alpha\alpha_{x}}{1+\alpha^2}\right)dx}\left(h(y)-\int \frac{H\alpha}{(1+\alpha^2)^{1/2}}e^{\int\left(2\alpha+\frac{\alpha\alpha_{x}}{1+\alpha^2}\right)dx}dx\right)\\
&=\frac{e^{-\int2\alpha dx}}{(1+\alpha^2)^{\frac{1}{2}}}\left(h(y)-\int H\alpha \left(e^{\int 2\alpha dx}\right) dx\right),
\end{split}
\end{equation}
for some function $h(y)$ in $y$ and $\int H\alpha \left(e^{\int 2\alpha dx}\right)dx$ is an anti-derivative of $H\alpha \left(e^{\int 2\alpha dx}\right)$ with respect to $x$. From (\ref{fofora}) and (\ref{foforb}), we conclude that the first fundamental form $I$ (or $a$ and $b$) is determined by $\alpha$ and $H$, up to two functions $k(y)$ and $h(y)$. We are thus ready to prove a more useful version of the fundamental theorem for surfaces (see Theorem \ref{main021}). 

\subsection{The proof of Theorem \ref{main021} with arbitrary $\alpha$ and $H$}

We define a Riemannian metric on $U$ by regarding $\{\frac{\partial}{\partial x},a(x,y)\frac{\partial}{\partial x}+b(x,y)\frac{\partial}{\partial y}\}$ as an orthonormal frame field, where $a$ and $b$ are specified as (\ref{foforb}) and (\ref{fofora}) with $h(y)$ and $k(y)$ given in (\ref{intcon01}).
Then it is easy to see that $\alpha$ and  $H$, together with $a$ and $b$, satisfy the integrability condition (\ref{intcon1}), and hence, by the fundamental theorem for surfaces (see Theorem \ref{fudthm}), $U$ can be embedded uniquely as a surface with $H$ and $\alpha$ as its $p$-mean curvature and $\alpha$-function, respectively. In addition, the characteristic direction $\hat{e}_{1}=\frac{\partial}{\partial x}$ and $\hat{e}_{2}=a(x,y)\frac{\partial}{\partial x}+b(x,y)\frac{\partial}{\partial y}$ define the induced metric on the embedded surface, as desired (see Theorem \ref{fudthm} for the original version).

\subsection{The Transformation law of invariants} 
 First of all, we compute  the transformation law of compatible coordinate systems. Let $(x,y)$ and $(\tilde{x},\tilde{y})$ be two compatible coordinate systems and $\phi$ a coordinate transformation, i.e., $(\tilde{x},\tilde{y})=\phi(x,y)$. Then we have $\phi_{*}\frac{\partial}{\partial x}=\frac{\partial}{\partial \tilde{x}}$, which means that
\begin{equation*}
[\phi_{*}]\left(\begin{array}{c}1\\0\end{array}\right)=\left(\begin{array}{cc}\tilde{x}_{x}&\tilde{x}_{y}\\\tilde{y}_{x}&\tilde{y}_{y}\end{array}\right)\left(\begin{array}{c}1\\0\end{array}\right)=\left(\begin{array}{c}1\\0\end{array}\right),
\end{equation*}
where $[\phi_{*}]$ is the matrix representation of $\phi_{*}$ with respect to these two coordinate systems. We then have the coordinates transformation
\begin{equation}\label{cortra}
\tilde{x}=x+\Gamma(y),\ \ \ \tilde{y}=\Psi(y),
\end{equation}
for some functions $\Gamma(y)$ and $\Psi(y)$. Since $\det{[\phi_{*}]}=\Psi^{'}(y)$, we immediately have that $\Psi^{'}(y)\neq 0$ for all $y$. 
Next, we compute the transformation law of representations of the induced metrics. Suppose that the representations of the induced metric are, respectively, given by $a,b$ and $\tilde{a},\tilde{b}$, that is, $\hat{e}_{2}=a\frac{\partial}{\partial x}+b\frac{\partial}{\partial y}=\tilde{a}\frac{\partial}{\partial \tilde{x}}+\tilde{b}\frac{\partial}{\partial \tilde{y}}$.
By \eqref{cortra}, we have the following transformation law of the induced metric as
\begin{equation}\label{mettra}
\tilde{a}=a+b\Gamma'(y),\ \ \ \tilde{b}=b\Psi'(y).
\end{equation}
Notice that we have omitted the sign of pull-back $\phi^{*}$ on $\tilde{a}$ and $\tilde{b}$. Since the $p$-mean curvature and $\alpha$-function are function-type invariants, they transform by pull-back. 

\section{Constant $p$-mean curvature surfaces}  
In this section, we aim to prove Theorem \ref{main021} with constant $H$ and then provide a new tool to study the constant $p$-mean curvature surfaces. More precisely, we will indicate that one can convert the investigation of the constant $p$-mean curvature surfaces into the study of the so-called {\bf Codazzi-like} equation.

\subsection{The proof of Theorem \ref{main021} with constant $H$}
Let $\Sigma\subset H_{1}$ be a {\bf constant} $p$-mean curvature surface with $H=c$. 
Then, in terms of a compatible coordinate system $(U;x,y)$, the integrability condition (\ref{intcon1}) is reduced to
\begin{equation}\label{intcon2}
\begin{split}
-a_{x}+a\frac{b_{x}}{b}&=\frac{c\alpha}{(1+\alpha^2)^{1/2}},\\
-\frac{b_{x}}{b}&=2\alpha+\frac{\alpha\alpha_{x}}{1+\alpha^2},\\
\alpha_{xx}+6\alpha\alpha_{x}&+4\alpha^{3}+c^{2}\alpha=0.
\end{split}
\end{equation}
In other words, there exists the $\alpha$ satisfying the {\bf Codazzi-like} equation
\begin{equation}\label{Codeq}
\alpha_{xx}+6\alpha\alpha_{x}+4\alpha^{3}+c^{2}\alpha=0,
\end{equation}
which is a nonlinear ordinary differential equation. 
Conversely, given an arbitrary function $\alpha(x,y)$ on a coordinate neighborhood $(U,x,y)\subset\R^{2}$ which satisfies the {\bf Codazzi-like} equation \eqref{Codeq} (or \eqref{Codeq01}). It is easy to see that equation (\ref{Codeq}) is just the equation \eqref{intcon01} in the constant $p$-mean curvature cases. Namely, that $\alpha$ satisfies equation (\ref{Codeq}) means it satisfies (\ref{intcon01}) for arbitrary functions $h(y)$ and $k(y)$. Therefore, the embedding $\Sigma=X(U)$ in Theorem \ref{main021} is a {\bf constant} $p$-mean curvature surface with $H=c$, the given function $\alpha(x,y)$ as its $\alpha$-function, and $\hat{e}_{1}=\frac{\partial}{\partial x}$. Notice that, given (\ref{foforb}) and (\ref{fofora}), the constant $p$-mean curvature surface determined by the given function $\alpha$ usually is not unique. They depend on two functions $h(y)$ and $k(y)$ in $y$. The assertion is completed.

Throughout the rest of this paper, we will apply the tool to the subject of the $p$-minimal surfaces.

\section{Solutions to the {\it Li\'enard equation}}\label{sofsotl}
Since we bring in a strategy to study $p$-minimal surfaces using the understanding of the {\bf Codazzi-like} equation \eqref{lieq},
we focus on studying this equation in this section. First of all, we suppose that $\alpha$ is regarded as a function in $x$ and we want to discuss the solutions to the {\bf Codazzi-like} equation 
\begin{equation}\label{lieq}
\alpha_{xx}+6\alpha\alpha_{x}+4\alpha^{3}=0.
\end{equation}
This is actually one kind of the so-called {\it Li\'enard equation} \cite{Burton}. Derivation of explicit solutions (see Theorem \ref{main031}) to the equation \eqref{lieq} is given below.
 
 \subsection{The proof of Theorem \ref{main031}}
 
Using $v=\alpha'$ to see that \eqref{lieq} becomes 
\begin{equation}\label{abel2}
v\frac{dv}{d\alpha}+6\alpha v+4\alpha^{3}=0,
\end{equation}
which is the second kind of the Abel equation (cf. \cite{ZP94,PZ03}). 
Apparently, given $v\neq 0$, equation \eqref{abel2} can be written as
\begin{equation}\label{chara}
\frac{dv}{d\alpha}=-\frac{2\alpha (3v+2\alpha^{2})}{v}.\end{equation}
Denote $u=\frac{1}{v}$. Then \eqref{abel2} or \eqref{chara} becomes
\begin{equation}\label{abel3}
	\frac{du}{d\alpha}=6\alpha u^2+4\alpha^3u^3.
\end{equation}

We apply Chiellini's integrability condition (stated in \cite{LA28,CA31}) for the Abel equation to \eqref{abel3}, which is exactly integrated with $k=\frac{2}{9}\neq 0$. It can be checked that 
$$
\frac{d}{d\alpha}\left(\frac{4\alpha^3}{6\alpha}\right)=\frac{d}{d\alpha}\left(\frac{2}{3}\alpha^2\right)=\frac{4}{3}\alpha=\frac{2}{9}(6\alpha).
$$
Let $u=\frac{6\alpha}{4\alpha^3}\omega$, i.e., $\omega=\frac{2}{3}\alpha^2u$ and $\omega \neq 0$. The equation \eqref{abel3} turns out to be
\begin{equation}\label{abel4}
	\frac{d\omega}{d\alpha}=\frac{\omega}{\alpha}(2+9\omega+9\omega^2),
\end{equation}
which is a separable first-order ODE, i.e.,
\begin{equation}\label{abel5}
	\frac{d\omega}{\omega(2+9\omega+9\omega^2)}=\frac{d\alpha}{\alpha}.
\end{equation}
Using the method of partial fractions to integrate the left-hand side, we have 
\begin{equation}\label{abel6}
	\int \left(\frac{1}{2\omega}-\frac{3}{3\omega+1}+\frac{3}{2(3\omega+2)}\right)d\omega=\int \frac{d\alpha}{\alpha},
\end{equation}
which gives
\begin{equation}\label{abel7}
	\frac{1}{2}\left(\ln \frac{|\omega(3\omega+2)|}{|3\omega+1|^2}\right)=\ln |\alpha|+const.
\end{equation}
The implicit solution to \eqref{abel4} is then expressed as 
\begin{equation}\label{sol_w}
	\frac{\omega(3\omega+2)}{(3\omega+1)^2}=C\alpha^2,
\end{equation}  
provided that $\omega \neq 0,-\frac{1}{3},-\frac{2}{3}$, and $C$ is an arbitrary nonzero constant.
Hence, with the assumption $\alpha\neq 0$, we have 
\begin{equation}
	4\alpha^2(3C\alpha^2-1)u^2+4(3C\alpha^2-1)u+3C=0,
\end{equation}
or equivalently,
\begin{equation}
	3C(\alpha')^2+4(3C\alpha^2-1)\alpha'+4\alpha^2(3C\alpha^2-1)=0.
\end{equation}
This yields
\begin{equation}\label{1stode}
\begin{split}
	\alpha'&=\frac{2(1-3C\alpha^2)\pm2\sqrt{1-3C\alpha^2}}{3C}\\
	&=\frac{(1-2C\alpha^2)\pm\sqrt{1-2C\alpha^2}}{C},
	\end{split}
\end{equation}
for a nonzero constant $C$. Since we have assumed that $v=\alpha'$ is not zero, neither is $1-2C\alpha^2$. To obtain the general solutions, we proceed to solve equation \eqref{1stode} by means of the variable separable method. We rewrite \eqref{1stode} as
\begin{equation}\label{1stode1}
\begin{split}
\frac{dx}{C}&=\frac{d\alpha}{(1-2C\alpha^2)\pm\sqrt{1-2C\alpha^2}}\\
&=\frac{(1-2C\alpha^2)\mp\sqrt{1-2C\alpha^2}}{[(1-2C\alpha^2)\pm\sqrt{1-2C\alpha^2}][(1-2C\alpha^2)\mp\sqrt{1-2C\alpha^2}]}d\alpha\\
&=\frac{(1-2C\alpha^2)\mp\sqrt{1-2C\alpha^2}}{-2C\alpha^{2}(1-2C\alpha^2)}d\alpha\\
&=\left(\frac{-1}{2C\alpha^2}\pm\frac{1}{2C\alpha^2\sqrt{1-2C\alpha^2}}\right)d\alpha.
\end{split}
\end{equation}
{\bf Case I.} If $C<0$, we use the trigonometric substitution $\sqrt{2|C|}\alpha=\tan{\theta},\ -\frac{\pi}{2}<\theta<\frac{\pi}{2}$, to get
\begin{equation}\label{1stode2}
\int \frac{d\alpha}{2C\alpha^2\sqrt{1-2C\alpha^2}}=-\frac{\sqrt{1-2C\alpha^2}}{2C\alpha}+c_{1},
\end{equation}
for some $c_{1}\in \mathbb R$. Substituting \eqref{1stode2} into \eqref{1stode1}, we obtain
\begin{equation*}
\frac{x+c_{1}}{C}=\frac{1\mp\sqrt{1-2C\alpha^2}}{2C\alpha},
\end{equation*}
that is,
\begin{equation}\label{1stode3}
2\alpha(x+c_{1})-1=\mp\sqrt{1-2C\alpha^2},
\end{equation}
for some $c_{1}\in \mathbb R$. Taking the square of both sides and noticing that $\alpha\neq 0$, we obtain
\begin{equation}\label{1stode4}
\alpha(x)=\frac{x+c_{1}}{(x+c_{1})^{2}+c_{2}},
\end{equation}
for some $c_{1},\ c_{2}\in \mathbb R$ and $c_{2}<0$. If $\alpha$ satisfies $2\alpha(x+c_{1})-1=+\sqrt{1-2C\alpha^2}$ in \eqref{1stode3}, then we have $\alpha(x+c_{1})>0$, and hence \eqref{1stode4} implies $x+c_{1}<-\sqrt{|c_{2}|}$ or $\sqrt{|c_{2}|}<x+c_{1}$. On the other hand, if $\alpha$ satisfies $2\alpha(x+c_{1})-1=-\sqrt{1-2C\alpha^2}$, then we have $\alpha(x+c_{1})<0$, and we then obtain that $-\sqrt{|c_{2}|}<x+c_{1}<\sqrt{|c_{2}|}$ by \eqref{1stode4}.\\
{\bf Case II.} If $C>0$, we use the trigonometric substitution $\sqrt{2C}\alpha=\sin{\theta},\ -\frac{\pi}{2}<\theta<\frac{\pi}{2}$, to get
\begin{equation}\label{1stode5}
\int \frac{d\alpha}{2C\alpha^2\sqrt{1-2C\alpha^2}}=-\frac{\sqrt{1-2C\alpha^2}}{2C\alpha}+c_{1},
\end{equation}
for some $c_{1}\in \mathbb R$. Substituting \eqref{1stode5} into \eqref{1stode1}, we obtain
\begin{equation*}
\frac{x+c_{1}}{C}=\frac{1\mp\sqrt{1-2C\alpha^2}}{2C\alpha},
\end{equation*}
that is,
\begin{equation}\label{1stode6}
2\alpha(x+c_{1})-1=\mp\sqrt{1-2C\alpha^2},
\end{equation}
for some $c_{1}\in \mathbb R$. Taking the square of both sides and noticing that $\alpha\neq 0$, we obtain
\begin{equation}\label{1stode7}
\alpha(x)=\frac{x+c_{1}}{(x+c_{1})^{2}+c_{2}},
\end{equation}
for some $c_{1},\ c_{2}\in \mathbb R$ and $c_{2}>0$. 

As above, while we try to get the general solutions to \eqref{lieq}, we have assumed that $\alpha'\neq 0, \alpha\neq 0, \omega\neq 0, \omega\neq-\frac{1}{3}$ and $\omega\neq-\frac{2}{3}$. Now suppose $\alpha'=0$ on an open interval, then \eqref{lieq} immediately implies $\alpha=0$ on that interval. It is easy to see that $\omega=0$ is equivalent to $\alpha=0$. Finally, since $\omega=\frac{2}{3}\alpha^{2}u$, we see that
\begin{equation*}
\omega=-\frac{1}{3}\Leftrightarrow \alpha^{2}u=\frac{-1}{2} \Leftrightarrow \alpha'=-2\alpha^{2} \Leftrightarrow\alpha(x)=\frac{1}{2(x+c_{1})}.
\end{equation*}
Similarly, we have 
\begin{equation*}
\omega=-\frac{2}{3} \Leftrightarrow\alpha(x)=\frac{1}{(x+c_{1})},
\end{equation*}
for some $c_{1}\in R$. We hence complete the proof of Theorem \ref{main031}. A similar argument establishes the following proof. 

\subsection{The proof of Theorem \ref{consta}}
Using a similar transformation between $v,u$ and $\omega$, the ODE \eqref{Codeq01} is converted to be
\begin{equation}\label{consta5}
	\frac{d\omega}{\omega(2+9\omega+9\omega^2)}=\frac{4\alpha}{4\alpha^2+c^2}d\alpha \quad(\omega=\frac{4\alpha^3+c^2\alpha}{6\alpha}u),
\end{equation}
which has the implicit solution expressed as 
\begin{equation}\label{consta_w}
	\frac{\omega(3\omega+2)}{(3\omega+1)^2}=K(4\alpha^2+c^2),
\end{equation}  
provided that $\omega \neq 0,-\frac{1}{3},-\frac{2}{3}$, and $K=e^{const}$ is arbitrary nonzero constant. If $\omega =0,-\frac{1}{3}$ or $-\frac{2}{3}$, we have the trivial solution or the special solutions derived.
Hence, with the assumption $\alpha\neq 0$, we have 
\begin{equation}\label{consta_1stode}	\alpha'=\frac{1-3K(4\alpha^2+c^2)\pm\sqrt{1-3K(4\alpha^2+c^2)}}{6K},
\end{equation}
for a nonzero constant $K$. Since we have assumed that $v=\alpha'\neq 0$, neither is $1-3K(4\alpha^2+c^2)$. In order to obtain the general solutions, we rationalize \eqref{consta_1stode} as
\begin{equation}\label{consta_1stode1}
\begin{split}
\frac{dx}{6K}&=\frac{d\alpha}{(1-3K(4\alpha^2+c^2))\pm\sqrt{1-3K(4\alpha^2+c^2)}}\\
&=\left(\frac{-1}{3K(4\alpha^2+c^2)}\pm\frac{1}{3K(4\alpha^2+c^2)\sqrt{1-3K(4\alpha^2+c^2)}}\right)d\alpha.
\end{split}
\end{equation}
Next, we will deal with integrations of both sides.
Note that 
\begin{equation}
	1-3K(4\alpha^2+c^2)=(1-3Kc^2)-12K\alpha^2 \geq 0,
\end{equation}
which implies \[1-3Kc^2 \geq 12K\alpha^2 \geq 0.\] Moreover, that $K > 0$ implies $0\leq 3K(4\alpha^2+c^2) $. We further assume that $3K(4\alpha^2+c^2) \leq 1$.
\begin{itemize}
\item If $1-3Kc^2=0$, then $3Kc^2=1$, which yields $3K(4\alpha^2+c^2)=12K\alpha^2+1 >1$, a contradiction.

	\item If $1-3Kc^2>0$, then we let
\begin{equation}\label{subs}
	\sqrt{12K}\alpha=\sqrt{1-3Kc^2}\sin \theta.
\end{equation} 
This gives
\begin{equation}
	d\alpha =\sqrt{\frac{1-3Kc^2}{12K}}\cos \theta d\theta,\quad
	1-3K(4\alpha^2+c^2) = (1-3Kc^2)\cos^2\theta.
\end{equation}
It turns out that the integration of $\frac{1}{3K(4\alpha^2+c^2)\sqrt{1-3K(4\alpha^2+c^2)}}d\alpha$ becomes
\begin{equation}
	\frac{1}{\sqrt{12K}}\int \frac{d \theta}{1-(1-3Kc^2)\cos ^2\theta}=\frac{1}{\sqrt{12K}}\left(-\frac{\tan^{-1}\left(\frac{\tan \theta}{\sqrt{3Kc^2}}\right)}{\sqrt{3Kc^2}}\right)+const.
\end{equation}
 By \eqref{subs}, we have 
 \begin{equation}
 	\frac{1}{\sqrt{12K}}\left(\frac{\tan^{-1}\left(\frac{2\alpha}{|c|\sqrt{1-3K(4\alpha^2+c^2)}}\right)}{\sqrt{3Kc^2}}\right)+const..
 \end{equation}
\end{itemize}

It is easy to see that
$$
\int \frac{d\alpha}{3K(4\alpha^2+c^2)}=\frac{1}{3K}\frac{\tan^{-1} \left(\frac{2\alpha}{c}\right)}{2c}+const..
$$
Therefore, \eqref{consta_1stode1} shows 
\begin{equation}\label{consta_1stode3}
\frac{x}{6K}=-\frac{1}{3K}\frac{\tan^{-1} \left(\frac{2\alpha}{c}\right)}{2c}\pm \frac{1}{6K|c|}\left(\tan^{-1}\left(\frac{2\alpha}{|c|\sqrt{1-3K(4\alpha^2+c^2)}}\right)\right)+C,
\end{equation}
where $C$ is a constant. Note that the sum/difference identity of $\tan(x\pm y)$ implies
\begin{equation}
\tan^{-1}A\pm\tan^{-1}B=\tan^{-1}\left(\frac{A\pm B}{1\mp AB}\right)+n\pi, \quad n\in\mathbb Z.
\end{equation}
Assuming $c>0$, \eqref{consta_1stode3} becomes
\[
-cx+C_1+n\pi=\tan^{-1}\left(\frac{2c\alpha\sqrt{1-3K(4\alpha^2+c^2)}\mp 2c\alpha}{c^2\sqrt{1-3K(4\alpha^2+c^2)}\pm 4\alpha^2}\right),
\]
which is
\[
\tan(-cx-c_1)=\frac{2c\alpha\sqrt{1-3K(4\alpha^2+c^2)}\mp 2c\alpha}{c^2\sqrt{1-3K(4\alpha^2+c^2)}\pm 4\alpha^2}.
\]
After a direct computation, the square sum of the denominator and numerator is equal to \[(4\alpha^{2}+c^{2})^{2}(1-3c^{2}K),\] and hence we have
\[\begin{split}
\sin{(-cx-c_{1})}&=\frac{2\alpha c\sqrt{1-3K(4\alpha^{2}+c^{2})}\mp 2\alpha c}{(4\alpha^{2}+c^{2})\sqrt{(1-3c^{2}K)}},\\
 \mbox{and }\cos{(-cx-c_{1})}&=\frac{c^{2}\sqrt{1-3K(4\alpha^{2}+c^{2})}\pm 4\alpha^{2}}{(4\alpha^{2}+c^{2})\sqrt{(1-3c^{2}K)}}.
\end{split}\]
If we let \[c_{2}=\pm\frac{1}{\sqrt{1-3c^{2}K}},\]
and notice that $c=2\lambda$, we compute
\[\begin{split}
&-\lambda\frac{\sin{(-cx-c_{1})}}{c_{2}-\cos{(-cx-c_{1})}}\\
=&\alpha\frac{-2c^{2}\sqrt{1-3K(4\alpha^{2}+c^{2})}\pm2c^{2}}{2c_{2}(4\alpha^{2}+c^{2})\sqrt{1-3c^{2}K}-2c^{2}\sqrt{1-3K(4\alpha^{2}+c^{2})}\mp 8\alpha^{2}}\\
=&\alpha
\end{split}\]

If $K < 0$, then $1- 3K(4\alpha^2+c^2) \geq 0$ automatically. It is easy to see that either $1- 3K(4\alpha^2+c^2) =0$ or  $1- 3K(4\alpha^2+c^2) >0$. 
 If $1-3K(4\alpha^2+c^2)=0$, then $(1-3Kc^2)+(-12K\alpha^2)=0$, which yields $1-3Kc^2=0$ since both  $(1-3Kc^2)$ and $-12K\alpha^2$ are non-negative. We use \eqref{consta_1stode} to have $\alpha'=0$, i.e., $\alpha$ is a constant w.r.t. $x$, which gives $\alpha=0$.
For the later case, $1-3K(4\alpha^2+c^2)>0$, and we further assume that $1-3Kc^2>0$, otherwise, we get $\alpha=0$. Direction calculations give the same equation \eqref{consta_1stode3}.

\subsection{The phase plane}
We remark that when $c=0$ (the case of $p$-minimal surfaces), \eqref{Codeq} is one of the so-called {\it Li\'enard equation} \cite{Burton}
\begin{equation}\label{lieq1}
\alpha_{xx}=f(\alpha,\alpha_{x})=-(6\alpha\alpha_{x}+4\alpha^{3}).
\end{equation}
If we imagine a simple dynamical system consisting of a particle of unit mass moving on the $\alpha$-axis, and if $f(\alpha,\alpha_{x})$ is the force acting on it, then \eqref{lieq1} is the equation of motion. The values of $\alpha$ (position) and $\alpha_{x}$ (velocity), which at each instant characterize the state of the system, are called its phases, and the plane of the variables $\alpha$ and $\alpha_{x}$ is called the {\bf phase plane}.
Using $v=\alpha_{x}$ to see that \eqref{lieq1} can be replaced by the equivalent system
\begin{equation}\label{lieq2}
\left\{\begin{split}
\frac{d\alpha}{dx}&=v,\\
\frac{dv}{dx}&=-(6\alpha v+4\alpha^{3}).
\end{split}\right.
\end{equation}
In general, a solution of \eqref{lieq2} is a pair of functions $\alpha(x)$ and $v(x)$ defining a curve on the phase plane. It follows from the standard theory of ODE that if $x_{0}$ is any number and $(\alpha_{0},v_{0})$ is any point in the phase plane, then there exists a unique solution $(\alpha(x),v(x))$ of \eqref{lieq2} such that $\alpha(x_{0})=\alpha_{0}$ and $v(x_{0})=v_{0}$. If this solution $(\alpha(x),v(x))$ is not a constant, then it defines a curve on the phase plane called a path of the system; otherwise, it defines a critical point. All paths together with critical points form a directed singular foliation on the phase plane with critical points as singular points of the foliation, and each path lies in a leaf of the foliation. It is easy to see that this foliation is defined by the vector field $V=(v,-(6\alpha v+4\alpha^{3}))$ and $(0,0)$ is the only critical point. We express the direction field (or the directed singular foliation) as Figure \ref{dfield} (Figure \ref{dfieldc} for $c=1.5$) as follows.

\begin{figure}[ht]
    \begin{minipage}{.49\textwidth}
   \centering
   \includegraphics[scale=0.5]{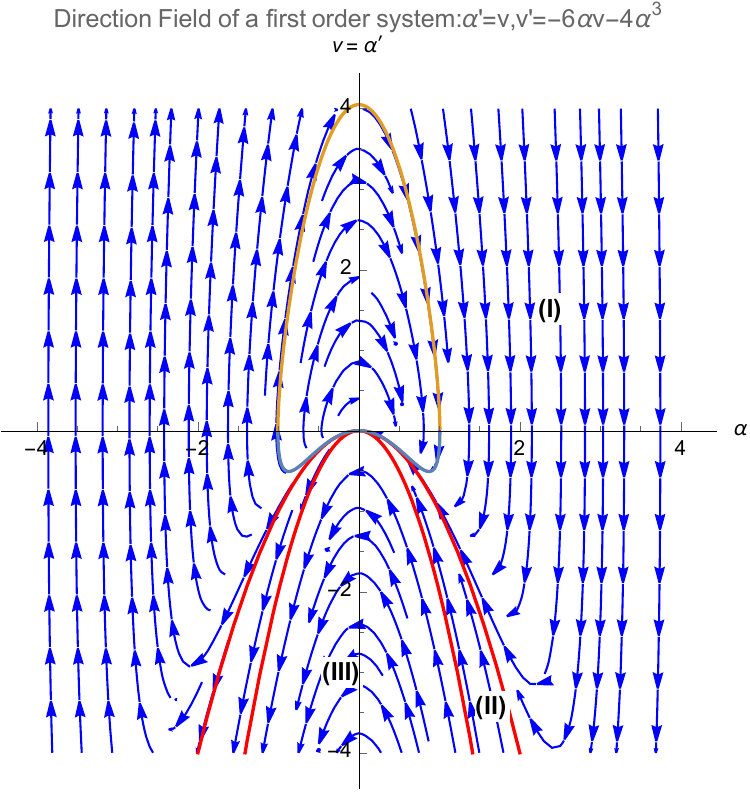}
   \caption{direction field $V$}
   \label{dfield}
      \end{minipage}
    \begin{minipage}{0.49\textwidth}
   \centering
   \includegraphics[scale=0.5]{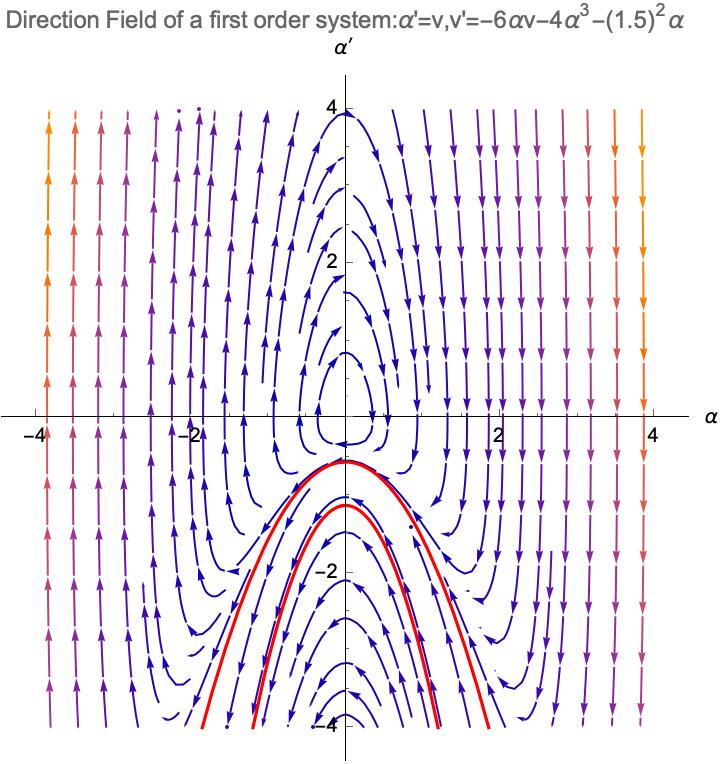}
   \caption{direction field for $c=\frac{3}{2}$}
   \label{dfieldc}
   \end{minipage}
\end{figure}

\section{The classification of constant $p$-mean curvature surfaces}
 \subsection{The classification of $p$-minimal surfaces} \label{claofpmi}
Theorem \ref{main031} suggests we divide (locally) the $p$-minimal surfaces into several classes. In terms of compatible coordinates $(x,y)$, the function $\alpha(x,y)$ is a solution
to the {\bf Codazzi-like} equation (\ref{lieq}) for any given $y$. By Theorem \ref{main031}, the function $\alpha(x,y)$ hence has one of the following forms of special types
\begin{equation*}
0,\ \frac{1}{x+c_{1}(y)},\ \frac{1}{2x+c_{1}(y)},
\end{equation*}
and general types
\begin{equation*}
\frac{x+c_{1}(y)}{(x+c_{1}(y))^{2}+c_{2}(y)},
\end{equation*}
where, instead of constants, both $c_{1}(y)$ and $c_{2}(y)$ are now functions of $y$. Notice that $c_{2}(y)\neq 0$ for all $y$. We now use the types of the function $\alpha(x,y)$ to define the types of $p$-minimal surface as follows.

\begin{de}\label{cladef} 
Locally, we say that a $p$-minimal surface is
\begin{enumerate}
\item {\bf vertical} if $\alpha$ vanishes (i.e., $\alpha(x,y)=0$ for all $x,y$).
\item of {\bf special type I} if $\alpha=\frac{1}{x+c_{1}(y)}$.
\item of {\bf special type II} if $\alpha=\frac{1}{2x+c_{1}(y)}$.
\item of {\bf general type} if $\alpha=\frac{x+c_{1}(y)}{(x+c_{1}(y))^{2}+c_{2}(y)}$ with $c_{2}(y)\neq 0$ for all $y$.
\end{enumerate}
\end{de}
We further divide $p$-minimal surfaces of general type into three classes as follows:
\begin{de}\label{cladef1}
We say that a $p$-minimal surface of {\bf general type} is
\begin{enumerate}
\item of {\bf type I} if $c_{2}(y)>0$ for all $y$.
\item of {\bf type II} if $c_{2}(y)<0$ for all $y$, and either $x<-c_{1}(y)-\sqrt{-c_{2}(y)}$ or $x>-c_{1}(y)+\sqrt{-c_{2}(y)}$.
\item of {\bf type III} if $c_{2}(y)<0$ for all $y$, and $-c_{1}(y)-\sqrt{-c_{2}(y)}<x<-c_{1}(y)+\sqrt{-c_{2}(y)}$.
\end{enumerate}
\end{de}
We notice that the {\bf type} is invariant under an action of a Heisenberg rigid motion and the regular part of a $p$-minimal surface $\Sigma\subset H_{1}$ is a union of these types of surfaces.
The corresponding paths of each type of $\alpha$ are marked on the phase plane (Figure \ref{dfield}). We express some basic facts about $p$-minimal surfaces with type as follows.
\begin{itemize}
\item If $\alpha$ vanishes, then it is part of a vertical plane.
\item The two concave downward parabolas represent 
$\alpha=\frac{1}{x+c_{1}}, \frac{1}{2x+c_{1}}$ respectively. The one for $\alpha=\frac{1}{x+c_{1}}$ is on top of the one for $\alpha=\frac{1}{2x+c_{1}}$. For surfaces of {\bf special type I}, we have that
\begin{equation}
\begin{split}
\alpha\rightarrow\left\{\begin{array}{rl}\infty,&\ \textrm{if}\ x\rightarrow-c_{1}\ \ \textrm{from the right},\\
-\infty,&\ \textrm{if}\ x\rightarrow-c_{1}\ \ \textrm{from the left};
\end{array}\right.
\end{split}
\end{equation}
and, for surfaces of {\bf special type II}, we have that
\begin{equation}
\begin{split}
\alpha\rightarrow\left\{\begin{array}{rl}\infty,&\ \textrm{if}\ x\rightarrow\frac{-c_{1}}{2}\ \ \textrm{from the right},\\
-\infty,&\ \textrm{if}\ x\rightarrow\frac{-c_{1}}{2}\ \ \textrm{from the left};
\end{array}\right.
\end{split}
\end{equation}
\item The closed curves with the origin removed correspond to the family of solutions
\begin{equation*}
\alpha(x)=\frac{x+c_{1}}{(x+c_{1})^{2}+c_{2}},	
\end{equation*}
where $c_1,c_2$ are constants and $c_2>0$, which are of {\bf type I}. There exists a zero for $\alpha$-function at $x=-c_{1}$. For surfaces of {\bf type I}, we have $|\alpha|\leq\frac{1}{2\sqrt{c_{2}}}$, so $\alpha$ is a bounded function for each fixed $y$, that is, along each path on the phase plane, $\alpha$ is bounded. Therefore, there are no singular points for surfaces of {\bf type I}.

\item The curves in between the two concave downward parabolas are of {\bf type II}. The $\alpha$-function of {\bf type II} does not have any zeros. For surfaces of {\bf type II}, it can be checked that
\begin{equation}
\begin{split}
\alpha\rightarrow\left\{\begin{array}{rl}\infty,&\ \textrm{if}\ x\rightarrow-c_{1}+\sqrt{-c_{2}}\ \ \textrm{from the right},\\
-\infty,&\ \textrm{if}\ x\rightarrow-c_{1}-\sqrt{-c_{2}}\ \ \textrm{from the left}.
\end{array}\right.
\end{split}
\end{equation}
\item The curves beneath the lower concave downward parabolas are of {\bf type III}. There exists a zero for $\alpha$-function at $x=-c_{1}$. For surfaces of {\bf type III}, we have
\begin{equation}
\begin{split}
\alpha\rightarrow\left\{\begin{array}{rl}-\infty,&\ \textrm{if}\ x\rightarrow-c_{1}+\sqrt{-c_{2}}\ \ \textrm{from the left},\\
\infty,&\ \textrm{if}\ x\rightarrow-c_{1}-\sqrt{-c_{2}}\ \ \textrm{from the right}.
\end{array}\right.
\end{split}
\end{equation}
\end{itemize}
\begin{prop}\label{indmetin1}
Suppose $\alpha(x,y)=\frac{x+c_{1}(y)}{(x+c_{1}(y))^{2}+c_{2}(y)}$, which is of {\bf general type}. Then the  explicit formula for the induced metric on a
$p$-minimal surface with this $\alpha$ as its $\alpha$-function is given by
\begin{equation}\label{foab42}
a=\frac{|\alpha|h(y)}{|x+c_{1}(y)|\sqrt{1+\alpha^{2}}},\ \ b=\frac{|\alpha|e^{k(y)}}{|x+c_{1}(y)|\sqrt{1+\alpha^{2}}},
\end{equation}
for some functions $h(y)$ and $k(y)$.
\end{prop}
\begin{proof}If $\alpha=\frac{x+c_{1}(y)}{(x+c_{1}(y))^{2}+c_{2}(y)}$, we choose $\ln{|(x+c_{1}(y))^{2}+c_{2}(y)|}$ as an anti-derivative of $2\alpha$ with respect to $x$. Simple computations imply
\begin{equation*}
\begin{split}
\frac{e^{-\int2\alpha dx}}{(1+\alpha^2)^{\frac{1}{2}}}&=\frac{1}{\sqrt{(x+c_{1}(y))^{2}+((x+c_{1}(y))^{2}+c_{2}(y))^{2}}}\\
&=\frac{|\alpha|}{|x+c_{1}(y)|\sqrt{1+\alpha^{2}}}.
\end{split}
\end{equation*}
Substituting the above formula into (\ref{foforb}) and (\ref{fofora}), equation (\ref{foab42}) follows.
\end{proof}

Similarly, we have

\begin{prop}\label{indmetin2}
Suppose $\alpha(x,y)=\frac{1}{x+c_{1}(y)}$, which is of {\bf special type I}. Then the explicit formula for the induced metric on a
$p$-minimal surface with this $\alpha$ as its $\alpha$-function is given by
\begin{equation}\label{foab422}
a=\frac{\alpha^{2}h(y)}{\sqrt{1+\alpha^{2}}},\ \ b=\frac{\alpha^{2}e^{k(y)}}{\sqrt{1+\alpha^{2}}}.
\end{equation}
\end{prop}
\begin{proof}
To obtain (\ref{foab422}), we choose $2\ln{|x+c_{1}(y)|}$ as an anti-derivative of $2\alpha$ with respect to $x$.
\end{proof}

\begin{prop}\label{indmetin3}
Suppose $\alpha(x,y)=\frac{1}{2x+c_{1}(y)}$, which is of {\bf special type II}. Then the explicit formula for the induced metric on a
$p$-minimal surface with this $\alpha$ as its $\alpha$-function is given by
\begin{equation}\label{foab423}
a=\frac{|\alpha|h(y)}{\sqrt{1+\alpha^{2}}},\ \ b=\frac{|\alpha|e^{k(y)}}{\sqrt{1+\alpha^{2}}}.
\end{equation}
\end{prop}
\begin{proof}
To have (\ref{foab423}), we choose $\ln{|2x+c_{1}(y)|}$ as an anti-derivative of $2\alpha$ with respect to $x$.
\end{proof}

\subsection{Analogs for constant $p$-mean curvature surfaces}\label{norconstp}
If $H=c\neq 0$ (we assume that $c=2\lambda>0$), then the Codazzi-like equation \eqref{Codeq01} has the trivial solution $\alpha(x)=0$ and the general solution of the form
\begin{equation}\label{solofg}
\alpha(x)=\lambda\frac{\sin{(2\lambda x+c_{1}})}{c_{2}-\cos{(2\lambda x+c_{1})}},
\end{equation}
which depends on two constants $c_{1}$ and $c_{2}$.

\begin{prop}\label{metricform} For any $\alpha(x,y)=\lambda\frac{\sin{(2\lambda x+c_{1}(y)})}{c_{2}(y)-\cos{(2\lambda x+c_{1}(y))}}$, the explicit formula for the induced metric on a
constant $p$-mean curvature surface with $2\lambda$ as its $p$-mean curvature and this $\alpha$ as its $\alpha$-function is given by
\begin{equation}\label{meta}
a=-\frac{\frac{c}{2}}{(1+\alpha^{2})^{1/2}}+\frac{\frac{c}{2}c_{2}}{(c_{2}-\cos{(cx+c_{1})})(1+\alpha^{2})^{1/2}}+\frac{h(y)}{|c_{2}-\cos{(cx+c_{1})}|(1+\alpha^{2})^{1/2}},
\end{equation}
and
\begin{equation}\label{metb}
b=\frac{e^{k(y)}}{|c_{2}-\cos{(cx+c_{1})}|(1+\alpha^{2})^{1/2}},
\end{equation}
for some functions $h(y)$ and $k(y)$.
\end{prop}
\begin{proof}
Let $f(x)=\ln{|c_{2}-\cos{(cx+c_{1}}|}, c=2\lambda$. Then $f'(x)=2\alpha(x)$. If we choose $\int 2\alpha dx=f(x)$, it is easy to see 
\[e^{\int 2\alpha dx}=|c_{2}-\cos{(cx+c_{1})|},\]
and hence,
\begin{equation*}
\begin{split}
a&=\frac{e^{-\int 2\alpha dx}}{(1+\alpha^{2})^{1/2}}\left(h(y)-\int c\alpha(e^{\int 2\alpha dx}) dx\right)\\
&=\frac{1}{|c_{2}-\cos{(cx+c_{1})}|(1+\alpha^{2})^{1/2}}\left(h(y)-\int c\alpha|c_{2}-\cos{(cx+c_{1})}| dx\right),
\end{split}
\end{equation*}
where
\[\begin{split}
\int c\alpha|c_{2}-\cos{(cx+c_{1})}| dx
&=\int\frac{c^2}{2}\frac{\sin{(c x+c_{1}})}{c_{2}-\cos{(c x+c_{1})}}|c_{2}-\cos{(cx+c_{1})}|dx\\
&=\left\{\begin{array}{rl}
-\frac{c}{2}\cos{(cx+c_{1})},&c_{2}-\cos{(cx+c_{1})}>0\\
\frac{c}{2}\cos{(cx+c_{1})},&c_{2}-\cos{(cx+c_{1})}<0
\end{array}\right..
\end{split}\]
Direct computations imply 
\begin{equation}\label{}
\begin{split}
a&=\frac{h(y)}{|c_{2}-\cos{(cx+c_{1})}|(1+\alpha^{2})^{1/2}}+\frac{-\frac{c}{2}(c_{2}-\cos{(cx+c_{1})})+\frac{c}{2}c_{2}}{(c_{2}-\cos{(cx+c_{1})})(1+\alpha^{2})^{1/2}}\\
&=-\frac{\frac{c}{2}}{(1+\alpha^{2})^{1/2}}+\frac{\frac{c}{2}c_{2}}{(c_{2}-\cos{(cx+c_{1})})(1+\alpha^{2})^{1/2}}+\frac{h(y)}{|c_{2}-\cos{(cx+c_{1})}|(1+\alpha^{2})^{1/2}}
\end{split}
\end{equation}
and
\begin{equation}\label{}
b=\frac{e^{-\int 2\alpha dx}}{(1+\alpha^{2})^{1/2}}e^{k(y)}=\frac{e^{k(y)}}{|c_{2}-\cos{(cx+c_{1})}|(1+\alpha^{2})^{1/2}}.
\end{equation}
\end{proof}
\begin{exa}[Pansu Sphere]
The Pansu sphere  in \cite{CCHY18} is the union of graphs of $f$ and $-f$ given by
\[f(z)=\frac{1}{2\lambda^2}\left(\lambda|z|\sqrt{1-\lambda^2|z|^2}+\cos^{-1}(\lambda|z|)\right), |z|\leq \frac{1}{\lambda},\]
which has a constant $p$-mean curvature $c=2\lambda$.
Its $\alpha$-function is of the following form,
\[\alpha=\frac{\lambda\sin(2\lambda s)}{(1-\cos(2\lambda s))}=-\lambda\tan(\lambda s+\frac{\pi}{2}),\]
and the metric is given by
\begin{equation}\label{pansume}
a=\frac{-\lambda}{\sqrt{1+\alpha^2}} \mbox{ and } b=\frac{2\lambda^2}{\sqrt{1+\alpha^2}(1-\cos 2\lambda s)}.
\end{equation}
\end{exa}

 \section{A representation of constant $p$-mean curvature surfaces}\label{othocs}
  Let $\Sigma\subset H_{1}$ be a $p$-minimal surface. We define an {\bf orthogonal coordinate system} $(x,y)$ to be a {\bf compatible} coordinate system such that $a=0$, that is, $\hat{e}_{2}=b\frac{\partial}{\partial y}$.
 
 \begin{prop}
There always exists an orthogonal coordinate system around any regular point of a $p$-minimal surface $\Sigma$.
\end{prop}
 \begin{proof}
 Suppose that $p\in\Sigma$ is a regular point and $(x,y)$ is an arbitrary compatible coordinate system around $p$. Since $H=0$, equations (\ref{foforb}) and (\ref{fofora}) imply that the ratio $\frac{-a}{b}=\frac{-h(y)}{e^{g(y)}}$ is just a function of $y$. 
Now we define another compatible coordinates $(\tilde{x},\tilde{y})$ by
\begin{equation*}
(\tilde{x},\tilde{y})=(x+\Gamma(y), \Psi(y)),
\end{equation*}
for some functions $\Gamma(y)$ and $\Psi(y)$ such that $\Gamma'(y)=\frac{-a}{b}$. By the transformation law (\ref{mettra}) of the representation of the induced metric, we have $\tilde{a}=0$. This means that $(\tilde{x},\tilde{y})$ are orthogonal coordinates around $p$. 
\end{proof}

\subsection{The proof of Theorem \ref{main041}}
It will be suitable to choose an orthogonal coordinate system $(U; x,y)$ to study a $p$-minimal surface. One sees easily from \eqref{cortra} and \eqref{mettra} that any two orthogonal coordinate systems $(x,y)$ and $(\tilde{x},\tilde{y})$ are transformed by
\begin{equation}
\tilde{x}=x+C,\ \ \ \tilde{y}=\Psi(y),
\end{equation}
for a constant $C$ and a function $\Psi(y)$. That is, the orthogonal coordinate systems are determined, up to a constant $C$ on the coordinate $x$ and scaling on the coordinate $y$. The transformation law of the representation of the induced metric hence reduces to
\begin{equation}\label{mettra2}
\tilde{a}=a=0,\ \tilde{b}=b\Psi'(y).
\end{equation}
In terms of orthogonal coordinate systems, the integrability condition hence reads 
\begin{equation}\label{intcon3}
\begin{split}
-\frac{b_{x}}{b}&=2\alpha+\frac{\alpha\alpha_{x}}{1+\alpha^2},\\
\alpha_{xx}&+6\alpha\alpha_{x}+4\alpha^{3}=0.
\end{split}
\end{equation}
 Then the $\alpha$-function determines  the metric representation $b$, and hence a $p$-minimal surface, up to a positive function $e^{k(y)}$ as (\ref{foforb}) specified (or see (\ref{foab42}),(\ref{foab422}) and (\ref{foab423}) ). Therefore, from the transformation law of the induced metric (\ref{mettra2}), we are able to choose another orthogonal coordinate system $(\tilde{x},\tilde{y})=\phi(x,y)$ with $\Psi$ satisfying $e^{k(y)}\Psi^{'}=1$. 
 That is, we can further normalize $b$ such that $k(\tilde{y})=0$ for each type, no matter it is special or general. Here $k(\tilde{y})$ is the function $k$ in the numerator of $\tilde{b}$ (see \eqref{foab42},\eqref{foab422} and \eqref{foab423}) with $\alpha,x,y$ replaced by $\tilde{\alpha},\tilde{x},\tilde{y}$. In fact, for a general type (the special types are similar), it is possible to choose another orthogonal coordinate system $(\tilde{x},\tilde{y})=\phi(x,y)$ such that
 \begin{equation*}
 \phi^{*}\tilde{b}=\frac{|\alpha|}{|x+c_{1}(y)|\sqrt{1+\alpha^{2}}}.
 \end{equation*} 
One sees that such orthogonal coordinate systems are unique up to a translation on the two variables $x$ and $y$. In other words, there are constants $C_{1}$ and $C_{2}$ such that 
 \begin{equation}
(\tilde{x},\tilde{y})=\phi(x,y)=(x+C_{1},y+C_{2}).
 \end{equation}
We call such an orthogonal coordinate system a {\bf normal coordinate system}. Indeed, since $\phi^{*}\tilde{\alpha}=\alpha$, Definition \ref{cladef} indicates the following transformation law for $c_{1}(y)$ and $c_{2}(y)$ functions
\begin{equation}
\begin{split}
\tilde{c}_{1}(\tilde{y})&=c_{1}(\tilde{y}-C_{2})-C_{1},\textrm{ for special type I},\\
\tilde{c}_{1}(\tilde{y})&=c_{1}(\tilde{y}-C_{2})-2C_{1},\textrm{ for special type II, and}\\
\tilde{c}_{1}(\tilde{y})&=c_{1}(\tilde{y}-C_{2})-C_{1},\ \tilde{c}_{2}(\tilde{y})=c_{2}(\tilde{y}-C_{2}),\textrm{ for general type},\\
\end{split}
\end{equation}
where $\tilde{c}_{1}(\tilde{y})$ and $\tilde{c}_{2}(\tilde{y})$ are with respect to $\tilde{\alpha}$. Namely, $c_{2}(y)$ is unique up to a translation on $y$, and 
$c_{1}(y)$ is unique up to a translation on $y$ and its image as well. We denote these two unique functions $c_{1}(y)$ and $c_{2}(y)$ by $\zeta_{1}(y)$ and $\zeta_{2}(y)$, respectively. We then complete the proof of Theorem \ref{main041}.\\

 Both two functions $\zeta_{1}(y)$ and $\zeta_{2}(y)$ are invariants under a Heisenberg rigid motion. Therefore we call them $\zeta_{1}$- and $\zeta_{2}$-invariants, respectively. In terms of $\zeta_{1}$ and $\zeta_{2}$, we thus have the version of the fundamental theorem for $p$-minimal surfaces in $H_{1}$ (Theorem \ref{main051}).

\subsection{The proof of Theorem \ref{main051}}
Given $\zeta_{1}(y)$, for (1) in Theorem \ref{main051}, we define $\alpha,a,b$ on $U$ by
\[\alpha=\frac{1}{x+\zeta_{1}(y)}, a=0\ \textrm{and}\ b=\frac{\alpha^2}{\sqrt{1+\alpha^{2}}}.\]
Notice that $(e,f)$ needs to be chosen so that $(e,f)\times(c,d)$ {\bf does not contain the zero set} of $x+\zeta_{1}(y)$.
Then they satisfy the integrability condition \eqref{intcon2} with $c=0$, and hence $U$ together with $\alpha,a,b$ can be embedded into $H_{1}$ to be a $p$-minimal surface with $\alpha$ as its $\alpha$-function, and the induced metric  $a,b$. Moreover the characteristic direction $e_{1}=\frac{\partial}{\partial x}$. From the type of $\alpha$, this minimal $p$-surface is of special type I. In view of $a=0$ and $b=\frac{\alpha^2}{\sqrt{1+\alpha^{2}}}$, we see that the coordinates $(x,y)$ are a normal coordinate system. Therefore $\zeta_{1}(y)$ is the corresponding $\zeta_{1}$-invariant. The uniqueness follows from the fundamental theorem for surfaces in $H_{1}$ or theorem \ref{main041}. This completes the proof of (1) for the special type I. Both proofs of (1) for the special type II and of (3) are similar with $(e,f)$ chosen according to their types. For the special type II, note that $(e,f)$ needs be chosen so that $(e,f)\times(c,d)$ {\bf does not contain the zero set} of $2x+\zeta_{1}(y)$, and we define $\alpha,a,b$ on $U$ by
\[\alpha=\frac{1}{2x+\zeta_{1}(y)}, a=0\ \textrm{and}\ b=\frac{|\alpha|}{\sqrt{1+\alpha^{2}}}.\]
For (3), we see that $(e,f)$ needs be chosen so that $(e,f)\times(c,d)$ {\bf contains no the zero set} of $(x+\zeta_{1}(y))^{2}+\zeta_{2}(y)$, and we define $\alpha,a,b$ on $U$ by
\[\alpha=\frac{x+\zeta_{1}(y)}{(x+\zeta_{1}(y))^{2}+\zeta_{2}(y)}, a=0\ \textrm{and}\ b=\frac{|\alpha|}{|x+\zeta_{1}(y)|\sqrt{1+\alpha^{2}}}.\]
Thus we complete the proof of Theorem \ref{main051}.\\

We remark that, in terms of normal coordinates $(x,y)$, the co-frame formula \eqref{indcof} reads
\begin{equation}\label{indcof1}
\begin{split}
\hat{\omega}^{1}&=dx-\frac{a}{b}dy=dx,\\
\hat{\omega}^{2}&=\frac{1}{b}dy,
\end{split}
\end{equation}
and hence the induced metric $I$ (the first fundamental form) reads
\begin{equation}\label{indmet}
\begin{split}
I&=\hat{\omega}^{1}\otimes\hat{\omega}^{1}+\hat{\omega}^{2}\otimes\hat{\omega}^{2}=dx\otimes dx+\frac{1}{b^2}dy\otimes dy,\\
&=\left\{\begin{array}{ll}
dx\otimes dx+\big[(x+\zeta_{1}(y))^{2}+(x+\zeta_{1}(y))^{4}\big]dy\otimes dy,&\textrm{for {\bf special type I}},\\
&\\
dx\otimes dx+\big[1+(2x+\zeta_{1}(y))^{2}\big]dy\otimes dy,&\textrm{for {\bf special type II}},\\
&\\
dx\otimes dx+\big[(x+\zeta_{1}(y))^{2}+[(x+\zeta_{1}(y))^{2}+\zeta_{2}(y)]^{2}\big]dy\otimes dy,&\textrm{for {\bf general type}}.\\
\end{array}\right.
\end{split}
\end{equation}
From \eqref{indmet}, we see immediately that the induced metric $I$ degenerates on the singular set $\{(x,y)\ |\ x+\zeta_{1}(y)=0\}$ for surfaces of special type I. Therefore, it cannot extend smoothly through the singular set. On the other hand, this phenomenon does not happen for both special type II and general type.

\subsection{The maximal $p$-minimal surfaces and the proof of Theorem \ref{main061}}\label{mami}
From the proof of Theorem \ref{main051}, it is clear to see that 
\begin{itemize}
\item for the special type I, the open rectangle $U$ in (1) of Theorem \ref{main051} can be extended to be either 
\begin{equation}
\begin{split}
U_{I}^{-}&=\{(x,y)\in\R^2\ |\ y\in(c,d),\ x+\zeta_{1}(y)<0\},\ \textrm{or}\\
U_{I}^{+}&=\{(x,y)\in\R^2\ |\ y\in(c,d),\ x+\zeta_{1}(y>0\},
\end{split}
\end{equation}
which depends on that $U$ is originally contained in $U_{I}^{-}$ or $U_{I}^{+}$. Notice that the embedding $X$ might be just extended to be an immersion. Since both $U_{I}^{-}$ and $U_{I}^{+}$ are connected and simply connected, the immersion $X$ is unique, up to a Heisenberg rigid motion. We denote these two $p$-minimal surfaces of {\bf special type I} by $S_{I}^{-}(\zeta_{1})=X(U_{I}^{-})$ and $S_{I}^{+}(\zeta_{1})=X(U_{I}^{+})$. From \eqref{indmet}, we see that the induced metric $I$ degenerates on the singular set $\{(x,y)\in\R^2\ |\ y\in(c,d),\ x+\zeta_{1}(y)=0\}$.
\item for the special type II, the open rectangle $U$ in (2) of Theorem \ref{main051} can be extended to be either 
\begin{equation}
\begin{split}
U_{II}^{-}&=\{(x,y)\in\R^2\ |\ y\in(c,d),\ 2x+\zeta_{1}(y)<0\},\ \textrm{or}\\
U_{II}^{+}&=\{(x,y)\in\R^2\ |\ y\in(c,d),\ 2x+\zeta_{1}(y>0\},
\end{split}
\end{equation}
which depends on the fact that $U$ is originally contained in $U_{II}^{-}$ or $U_{II}^{+}$. The embedding $X$ might be extended to be an immersion. Since both $U_{II}^{-}$ and $U_{II}^{+}$ are connected and simply connected, the immersion $X$ is unique, up to a Heisenberg rigid motion. We denote these two $p$-minimal surfaces of {\bf special type II} by $S_{II}^{-}(\zeta_{1})=X(U_{II}^{-})$ and $S_{II}^{+}(\zeta_{1})=X(U_{II}^{+})$.
\item when $\zeta_{2}(y)>0$ for all $y\in(c,d)$, since there exist no zeros of $(x+\zeta_{1}(y))^{2}+\zeta_{2}(y)=0$, the open rectangle $U$ in (3) can be extended to be the product space 
\[V_{I}=\R\times(c,d).\] 
Since the extended immersion $X$ is unique, up to a Heisenberg rigid motion, we denote the $p$-minimal surface of {\bf type I} by $\Sigma_{I}(\zeta_{1},\zeta_{2})=X(V_{I})$.
\item when $\zeta_{2}(y)<0$ for all $y\in(c,d)$, since the zero set $(x+\zeta_{1}(y))^{2}+\zeta_{2}(y)=0$ consists of two separated curves defined by $x+\zeta_{1}(y)+\sqrt{-\zeta_{2}(y)}=0$ and $x+\zeta_{1}(y)-\sqrt{-\zeta_{2}(y)}=0$, respectively, the open rectangle $U$ in (3)  can be extended to be one of the following three domains:
\begin{equation}
\begin{split}
V_{II}^{-}&=\{(x,y)\in\R^2\ |\ y\in(c,d), x<-\zeta_{1}(y)-\sqrt{-\zeta_{2}(y)}\},\\
 V_{II}^{+}&=\{(x,y)\in\R^2\ |\ y\in(c,d), x>-\zeta_{1}(y)+\sqrt{-\zeta_{2}(y)}\},\ \ \textrm{and}\\
V_{III}&=\{(x,y)\in\R^2\ |\ y\in(c,d), -\zeta_{1}(y)-\sqrt{-\zeta_{2}(y)}<x<-\zeta_{1}(y)+\sqrt{-\zeta_{2}(y)}\},
\end{split}
\end{equation}
Since the extended immersion $X$ is unique, up to a Heisenberg rigid motion, we denote these two $p$-minimal surfaces of {\bf type II} by $\Sigma_{II}^{-}(\zeta_{1},\zeta_{2})=X(V_{II}^{-})$ and $\Sigma_{II}^{+}(\zeta_{1},\zeta_{2})=X(V_{II}^{+})$, and the $p$-minimal surface of {\bf type III} by $\Sigma_{III}(\zeta_{1},\zeta_{2})=X(V_{III})$.
\end{itemize}

We see that $\zeta_{1}$-invariant is the only invariant for $p$-minimal surfaces of special types, and $\zeta_{1}$- and $\zeta_{2}$-invariants are the only two invariants for general type. From the above, Theorem \ref{main061} immediately follows.

\subsection{Symmetric $p$-minimal surfaces}
 A $p$-minimal surface is called symmetric if $\zeta_{1}$-invariant is a constant for the special types; both $\zeta_{1}$- and 
$\zeta_{2}$-invariants are constants for the general types. Since $\zeta_{1}$, up to a translation on its image, is an invariant, we presently have

\begin{thm}\label{main06}
All symmetric $p$-minimal surfaces of the same special type are locally congruent to one another, whereas for the general type, locally there is 
 a family of symmetric $p$-minimal surfaces, depending on a parameter on $\mathbb R$.
\end{thm}

\subsection{The normalization of constant $p$-mean curvature surfaces}\label{norconstp1}
Unlike $p$-minimal surfaces, Proposition \ref{metricform} indicates that normalizing $a$ to be zero is not applicable. Therefore, we would like to normalize $a$ and $b$ so that they look like the induced metric of the Pansu sphere given in \eqref{pansume} as possible as we can. Let $(\tilde{x},\tilde{y})=(x+\Gamma(y),\Psi(y))$ be another compatible coordinates. The transformation laws imply
\begin{equation}
\begin{split}
\tilde{a}&=a+b\Gamma'(y)\\
&=-\frac{\frac{c}{2}}{(1+\alpha^{2})^{1/2}}+\frac{\pm(\lambda c_{2})+h(y)+e^{k(y)}\Gamma'(y)}{|c_{2}-\cos{(cx+c_{1})}|(1+\alpha^{2})^{1/2}},
\end{split}
\end{equation} 
where we have the sign ''$-$'' if $c_{2}-\cos{(cx+c_{1})}<0$. Therefore, if we can choose $\Gamma(y)$ such that $\pm(\lambda c_{2})+h(y)+e^{k(y)}\Gamma'(y)=0$, we have 
\begin{equation}
\tilde{a}=-\frac{\lambda}{(1+\alpha^{2})^{1/2}}.
\end{equation}
Similarly,
\begin{equation}
\tilde{b}=b\Psi'(y)=\frac{e^{k(y)}\Psi'(y)}{|c_{2}-\cos{(cx+c_{1})}|(1+\alpha^{2})^{1/2}},
\end{equation} 
which implies that if $\Psi(y)$ is chosen to be $e^{k(y)}\Psi'(y)=2\lambda^{2}$, we obtain
\begin{equation}
\tilde{b}=\frac{2\lambda^{2}}{|c_{2}-\cos{(cx+c_{1})}|(1+\alpha^{2})^{1/2}},
\end{equation}
 It is easy to see that such a coordinate system $(\tilde{x},\tilde{y})$ is unique up to a translation, i.e., $(\tilde{x},\tilde{y})=(x+C_{1},y+C_{2})$ for some constants $C_{1},C_{2}$. We call it {\it the normal coordinate system}.

\begin{thm}\label{cominv}
In normal coordinates $(x,y)$, the functions $c_{1}(y)$ and $c_{2}(y)$ in the expression $\alpha(x)=\lambda\frac{\sin{(2\lambda x+c_{1}})}{c_{2}-\cos{(2\lambda x+c_{1})}}$ are unique in the following sense: up to a translation on $y$, $c_{2}(y)$ is unique up to a sign; and $c_{1}(y)$ is unique up to a constant.
We denote these two unique functions by
\[\zeta_{1}(y)=c_{1}(y),\ \zeta_{2}(y)=c_{2}(y).\]
Therefore, $\{\zeta_{1}(y), \zeta_{2}(y)\}$ is a {\bf complete} set of invariants for those surfaces ($\alpha$ not vanishing).
\end{thm}

\section{Examples of $p$-minimal surfaces}\label{examples}

\subsection{Examples of special type I}\label{model1}
The following is a family of $p$-minimal surfaces. They are defined by the graphs of 
\begin{equation}
\label{pmin1}
u=Ax+By+C,
\end{equation}
for some real constants $A,B$ and $C$. It is easy to see that $(-B,A,C)$ or $(x,y)=(-B,A)$ is the only singular point of the graph of $u=Ax+By+C$.
\begin{lem}\label{le01}
The graph defined by \eqref{pmin1} is congruent to the graph of $u=0$.  
\end{lem}
\begin{proof}
After the action of the left translation by $(B,-A,-C)$, we have
\begin{equation*}
\begin{split}
(B,-A,-C)(x,y,u)&=(x+B,y-A,u-C-Ax-By)\\
&=(x+B,y-A,0).
\end{split}
\end{equation*}
This completes the proof.
\end{proof}

\begin{exa}\label{exa72}
The p-minimal surface defined by the graph of $u=0$ corresponds to $\alpha=\frac{1}{r}$, where $r=\sqrt{x^2+y^2}$. Indeed, let us consider a surface defined by 
\[X: (x,y) \rightarrow (x,y,0).\]	
The horizontal normal can be calculated as
\begin{equation}\label{spetye11}
e_2=\frac{(u_x-y)}{r}\overset{\circ}{e_1}+\frac{(u_y+x)}{r}\overset{\circ}{e_2}=\frac{-y}{\sqrt{x^2+y^2}}\overset{\circ}{e_1}+\frac{x}{\sqrt{x^2+y^2}}\overset{\circ}{e_2},
\end{equation}
and then
\begin{equation}\label{spetye12}
e_1=\frac{x}{\sqrt{x^2+y^2}}\overset{\circ}{e_1}+\frac{y}{\sqrt{x^2+y^2}}\overset{\circ}{e_2}=\frac{x}{\sqrt{x^2+y^2}}\frac{\partial }{\partial x}+\frac{y}{\sqrt{x^2+y^2}}\frac{\partial}{\partial y}.
\end{equation}
For the $\alpha$-function, making use of 
$$
\alpha e_2+T=\alpha(\frac{-y}{\sqrt{x^2+y^2}},\frac{x}{\sqrt{x^2+y^2}},\frac{-(x^2+y^2)}{\sqrt{x^2+y^2}})+(0,0,1).
$$
to derive $\alpha=\frac{1}{\sqrt{x^2+y^2}}$,
and hence $(0,0)$ is the only singular point. Notice that $(x,y)$ is not a compatible coordinate system. 
In terms of the polar coordinates $(r,\theta)$ with the coordinates transformation $x=r\cos{\theta}$ and $y=r\sin{\theta}$, that is, we consider the re-parametrization
\[X: (r,\theta) \rightarrow (r\cos{\theta},r\sin{\theta}, 0).\]
It represents 
\begin{equation}
X_{r}=(\cos{\theta},\sin{\theta},0),\  X_{\theta}=(-r\sin{\theta},r\cos{\theta},0).
\end{equation}
From \eqref{spetye12}, it is easy to see that $e_{1}=X_{r}=\frac{\partial}{\partial r}$, and thus $(r,\theta)$ is a compatible coordinate system. For the $\alpha$-function and the induced metric $a$ and $b$, we solve the equation
\[ \frac{\alpha e_{2}+T}{\sqrt{1+\alpha^{2}}}=aX_{r}+bX_{\theta}\]
to get $e_{2}=(-\sin{\theta},\cos{\theta},-r)$ from \eqref{spetye11}, and to obtain
$$
\alpha=\frac{1}{r}, a=0, b=\frac{\alpha^{2}}{\sqrt{1+\alpha^{2}}},
$$
which, from the formula of $b$, implies that the polar coordinates $(r,\theta)$ is a {\bf normal coordinate system}. Since $\alpha=\frac{1}{r}$, the surface $X$ is of special type I with the $\zeta_{1}(\theta)=0$ as $\zeta_{1}$-invariant, and hence $X$ is a symmetric $p$-minimal surface. 

\begin{figure}[ht]
 \begin{minipage}{.45\textwidth}
   \centering
   \includegraphics[scale=0.55]{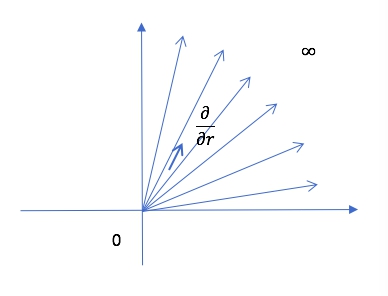}
   \caption{The characteristic direction field of $X$}
   \label{fig2}
   \end{minipage}
   \begin{minipage}{.45\textwidth}
      \centering
   \includegraphics[scale=0.4]{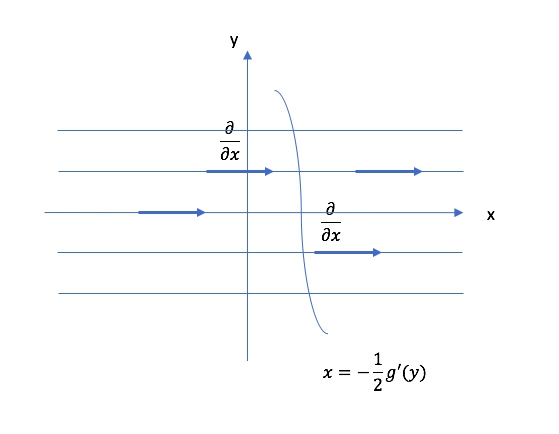}
   \caption{$e_{1}=\frac{\partial}{\partial x}$ for $x>-\frac{1}{2}g'(y)$}
   \label{fig3}
   \end{minipage}
\end{figure}
\end{exa}

Given Theorem \ref{main06}, we immediately have the following theorem.

\begin{thm}\label{main07}
A {\bf symmetric} $p$-minimal surface of {\bf special type I} is locally congruent to the graph of $u=0$.
\end{thm}
\begin{proof}
This is because that $\zeta_{1}$ is constant for a symmetric $p$-minimal surface of {\bf special type I} and the function $\zeta_{1}$, up to a constant, is a complete invariant.
\end{proof}

\begin{thm}\label{main08}
In terms of normal coordinates $(x,y)$, the induced metric (the first fundamental form) on a $p$-minimal surface of special type I degenerates on the singular set $\{(x,y)\ |\ x=-\zeta_{1}(y) \}$. Therefore if it is not symmetric, then it will never smoothly extend through the singular set.
\end{thm}
\begin{proof}
In terms of normal coordinates, a $p$-minimal surface of special type I is given by a function $\zeta_{1}(y)$, in which we have 
\[\alpha=\frac{1}{x+\zeta_{1}(y)},\ a=0,\ b=\frac{\alpha^{2}}{\sqrt{1+\alpha^{2}}}.\]
Therefore, 
\[\hat{e}_{2}=a\frac{\partial}{\partial x}+b\frac{\partial}{\partial y}=\frac{\alpha^{2}}{\sqrt{1+\alpha^{2}}}\frac{\partial}{\partial y}.\]
This is equivalent to that the induced metric $I$ is degenerate on the singular set on which $\alpha$ blows up. In fact, we have 
\begin{equation*}
I=\hat{\omega}^{1}\otimes\hat{\omega}^{1}+\hat{\omega}^{2}\otimes\hat{\omega}^{2}=dx\otimes dx+\left(\frac{1+\alpha^{2}}{\alpha^{4}}\right)dy\otimes dy,
\end{equation*}
where $(\hat{\omega}^{1},\hat{\omega}^{2})$ is the dual co-frame of $(\hat{e}_{1},\hat{e}_{2})$. If it is not symmetric, and suppose that it can be smoothly extended beyond the singular set, then the singular set is actually a singular curve  and the induced metric must be non-degenerate. This completes the proof.
\end{proof}

\subsection{Examples of special type II}\label{model2}
The other family of $p$-minimal surfaces is defined by the graph of 
\begin{equation}\label{pmin2}
u=-ABx^{2}+(A^{2}-B^{2})xy+ABy^{2}+g(-Bx+Ay),
\end{equation}
for some real constants $A$ and $B$ such that $A^{2}+B^{2}=1$ and $g\in C^{\infty}(\R)$.

\begin{lem}\label{le02}
The graph defined by (\ref{pmin2}) is congruent to the graph of $u=xy+g(y)$.  
\end{lem}
\begin{proof}
Since $A^{2}+B^{2}=1$, the matrix 
\begin{equation*}
\left(\begin{array}{rl}
A&B\\-B&A
\end{array}\right)
\end{equation*}
defines a rotation on $\mathbb R^{2}$. Let
\begin{equation*}
\left(\begin{array}{c}
X\\Y
\end{array}\right)=\left(\begin{array}{rl}
A&B\\-B&A
\end{array}\right)\left(\begin{array}{c}
x\\y
\end{array}\right),
\end{equation*}
we have
\begin{equation*}
\begin{split}
XY&=(Ax+By)(-Bx+Ay)\\
&=-ABx^{2}+(A^{2}-B^{2})xy+ABy^{2},
\end{split}
\end{equation*}
which implies
\begin{equation*}
u=XY+g(Y).
\end{equation*}
This completes the proof.
\end{proof}

\begin{exa}\label{exa76}
We now study the example of the graph of $u=xy+g(y)$. We consider a parametrization of the graph defined by 
\[
X: (x,y) \rightarrow (x,y,xy+g(y)),
\]
then we have 
\begin{equation}\label{comspeii1}
X_{x}=(1,0,y)=\mathring{e}_{1},\ \ 
X_{y}=(0,1,x+g'(y)).
\end{equation}
The horizontal normal $e_{2}$ is taken to be	
\begin{equation*}
e_2=\left\{\begin{array}{l}\frac{(u_x-y)}{D}\mathring{e}_{1}+\frac{(u_y+x)}{D}\mathring{e}_{2}\\
\\
-\frac{(u_x-y)}{D}\mathring{e}_{1}-\frac{(u_y+x)}{D}\mathring{e}_{2}\end{array}\right.\\
=\left\{\begin{array}{ll}\frac{2x+g'(y)}{D}\mathring{e}_{2}=\mathring{e}_{2}&,\ \textrm{if}\ 2x+g'(y)>0\\
\\
-\frac{2x+g'(y)}{D}\mathring{e}_{2}=\mathring{e}_{2}&,\ \textrm{if}\ 2x+g'(y)<0\end{array},\right.
\end{equation*}
where $D=|2x+g'(y)|$. Combining with \eqref{comspeii1}, one sees
\begin{equation}\label{comspeii2}
e_1=-Je_2=\mathring{e}_{1}=X_{x}=\frac{\partial}{\partial x}.
\end{equation}
 We proceed to compute the $\alpha$-function, $a$ and $b$ in terms of $(x,y)$, which is a compatible coordinate system.
These are obtained from $\frac{\alpha e_2+T}{\sqrt{1+\alpha^{2}}} =aX_{x}+bX_{y}$ immediately as follows. 
\begin{equation}\label{comspeii4}
a=0,\ b=\frac{\alpha}{\sqrt{1+\alpha^{2}}},\ \textrm{and}\ \ \alpha=\frac{1}{2x+g'(y)}.
\end{equation}

(i) Thus, from the formula of $b$, the coordinates $(x,y)$ is a {\bf normal coordinate system} on the part where $2x+g'(y)>0$, and we have $\zeta_{1}(y)=g'(y)$. It is easy to see that the graph of $u=xy+g(y)$, for $g\in C^{\infty}(\R)$, is just the maximal surface $S_{II}^{+}(g'(y))$ when we restrict to the domain $\{(x,y)\ |\ y\in\R,\ 2x+g'(y)>0\}$.

(ii) For the other part with $2x+g'(y)<0$, we have $b<0$. Therefore, instead of $(x,y)$, the new coordinate system $(\tilde{x},\tilde{y})=(x,-y)$ is a normal coordinate system (notice that the compatible coordinates are chosen so that $b>0$). The invariants $\alpha, a$ and $b$ read
\begin{equation}\label{comspeii5}
a=0,\ b=-\frac{\alpha}{\sqrt{1+\alpha^{2}}}>0,\ \textrm{and}\ \ \alpha=\frac{1}{2x+g'(-\tilde{y})}<0,
\end{equation}
and hence $\zeta_{1}(\tilde{y})=g'(-\tilde{y})$. Here $'$ means the derivative with respect to $y$.

(iii) For the other part with $2x+g'(y)<0$, we can say something more. If, instead of $\frac{\partial}{\partial x}$, we choose $-\frac{\partial}{\partial x}$ as the characteristic direction, that is, $e_{1}=-\frac{\partial}{\partial x}$, then the coordinates $(\tilde{x},\tilde{y})=(-x,-y)$ lead to the normal coordinate system for the part with $2x+g'(y)<0$. As a result, we consider the re-parametrization of the surface
\[
X: (\tilde{x},\tilde{y}) \rightarrow (-\tilde{x},-\tilde{y},\tilde{x}\tilde{y}+g(-\tilde{y})) 
\]
such that  $e_{1}=\frac{\partial}{\partial \tilde{x}}=X_{\tilde{x}}$. Similarly, from $\frac{\alpha e_2+T}{\sqrt{1+\alpha^{2}}} =aX_{\tilde{x}}+bX_{\tilde{y}}$, we have
\begin{equation}\label{comspeii6}
a=0,\ b=\frac{\alpha}{\sqrt{1+\alpha^{2}}}>0,\ \textrm{and}\ \ \alpha=\frac{1}{2\tilde{x}+\frac{\partial\tilde{g}}{\partial\tilde{y}}(\tilde{y})}>0, 
\end{equation}
where $\tilde{g}(\tilde{y})$ is defined by $\tilde{g}(\tilde{y})=g(-\tilde{y})$. Thus $(\tilde{x},\tilde{y})$ is the {\bf normal coordinate system} and we have $\zeta_{1}(\tilde{y})=\frac{\partial\tilde{g}}{\partial\tilde{y}}(\tilde{y})$.
\end{exa}

Let $R(g(y))$ and $L(g(y))$ be the part of the surface with $2x+g'(y)>0$ and $2x+g'(y)<0$, respectively. In terms of the notations defined in Subsection \ref{mami}, we see that $R(g(y))=S_{II}^{+}(g'(y))$ and $L(g(y))=S_{II}^{-}(g'(y))$. Then, comparing with 
\eqref{comspeii4} and \eqref{comspeii6}, we immediately have the following proposition, due to Theorem \ref{main051}.

\begin{prop}
The surface $L(g(-y))\left(\textrm{or}\ S_{II}^{-}(-g'(y))\right)$ and $R(g(y))\left(\textrm{or}\ S_{II}^{+}(g'(y))\right)$ are congruent to each other. They in fact differ by an action of the Heisenberg rigid motion
$(x,y,t)\rightarrow (-x,-y,t)$.
\end{prop}

\begin{thm}\label{main09}
Any $p$-minimal surface of special type II is locally a part of the surface defined by $u=xy+g(y)$ for some $g\in C^{\infty}(\mathbb R)$, up to a Heisenberg rigid motion. In addition, it is symmetric if and only if $g(y)$ is linear in the variable $y$.
Therefore, any {\bf symmetric} $p$-minimal surface of special type II is locally a part of the surface defined by the graph of $u=xy$, up to a Heisenberg rigid motion.
\end{thm}
\begin{proof}
Any $p$-minimal surface of special type II locally has the following normal representation
\begin{equation}
a=0,\ b=\frac{|\alpha|}{\sqrt{1+\alpha^{2}}},\ \ \alpha=\frac{1}{2x+\zeta_{1}(y)},
\end{equation}
in terms of normal coordinates $(x,y)$. Therefore, comparing with \eqref{comspeii4}, the proof is finished if we choose $g$ such that $g'(y)=\zeta_{1}(y)$. Moreover, it is symmetric if and only if $\zeta_{1}(y)=g'(y)=$ constant, i.e., $g$ is linear in $y$.
\end{proof}

\subsection{Examples of types I, II and III}\label{exagt}
We consider the surface $\Sigma\in H_{1}$ defined on $\R^2$ by 
\begin{equation}\label{gs}
X:(s,t) \rightarrow (x,y,z)=(s\cos \theta(t),s \sin \theta(t),t).	
 \end{equation} 
Then it can be calculated that $$X_s=(\cos\theta(t),\sin\theta(t),0),\quad X_t=(-s\theta'\sin\theta(t),s\theta'\cos\theta(t),1).$$
Notice that $\mathring{e}_{1}|_{(0,0,z)}=\frac{\partial}{\partial x}$, $\mathring{e}_{2}|_{(0,0,z)}=\frac{\partial}{\partial y}$ and $\theta(t)=\theta(z)$, i.e., $\theta$ is independent of $x$ and $y$. We rewrite $X_s$ as
\begin{equation}
	X_s=\cos\theta(t)\frac{\partial}{\partial x}+\sin\theta(t)\frac{\partial}{\partial y}
	=\cos\theta(t)\mathring{e}_{1}(X)+\sin\theta(t) \mathring{e}_{2}(X)\in\xi,
\end{equation}
which is a vector tangent to the contact plane. Then we choose $e_{1}=X_s$, and hence $e_2=Je_{1}=-\sin\theta(t)\mathring{e}_{1}(X)+\cos\theta(t) \mathring{e}_{2}(X)$, which yields
\begin{equation}
\begin{split}
	\bigtriangledown_{e_1}e_2&=-(e_1\theta(t))(\cos\theta(t)\mathring{e}_{1}(X)-\sin\theta(t)\mathring{e}_{2}(X))\\
	&=0,\ \left(\because e_{1}\theta(t)=\frac{d\theta(t)}{ds}=0\right).
	\end{split}
\end{equation}
This implies that such surface defined by \eqref{gs} has $p$-mean curvature $H=0$. We proceed to work out the $\alpha$-function $\alpha$, $a$ and $b$. By definition, it is a function satisfying $\alpha e_2+T \in T\Sigma$, that is,
\begin{equation}\label{comalp1}
\alpha(-\sin\theta\mathring{e}_{1}+\cos\theta\mathring{e}_{2})+T=EX_s+FX_t,
\end{equation}
 for some functions $E,F$. Similarly, we rewrite $X_s$ as a linear combination of $\mathring{e}_{1},\mathring{e}_{2}$ and $\frac{\partial}{\partial z}$, and we can express $X_t$ as
 \begin{equation}\label{comalp2}
 X_t=(-s\theta'\sin{\theta})\mathring{e}_{1}+(s\theta'\cos{\theta})\mathring{e}_{2}+(s^{2}\theta'+1)\frac{\partial}{\partial z}.
 \end{equation}
Combining  \eqref{comalp1} and \eqref{comalp2} and notice that $X_s=e_{1}$, we obtain that $E=0,\ F=\frac{1}{s^2\theta'(t)+1}$ and hence 
\begin{equation*}
\alpha=\frac{s\theta'(t)}{s^2\theta'(t)+1},\ a=0,\ b=\frac{F}{\sqrt{1+\alpha^2}}.
\end{equation*}
If $\theta'(t)=0$, then we have $\alpha=0$. However, if $\theta'(t)\neq 0$, then $\alpha$, $a$ and $b$ read
\begin{equation}\label{typfor}
\alpha=\frac{s}{s^2+\frac{1}{\theta'(t)}},\ a=0,\ b=\frac{\alpha}{s\sqrt{1+\alpha^{2}}}\frac{1}{\theta'(t)},
\end{equation}
which means that $(s,t)$ is an orthogonal coordinate system, but not normal. In particular, if $\theta'(t)>0$ for all $t$, then one sees that the $p$-minimal surface has no singularities. From (\ref{typfor}), we conclude that
this surface is of {\bf type I} if $\theta^{'}(t)>0$ for all $t$. On the other hand, if $\theta^{'}(t)<0$ for all $t$, then it is either of 
{\bf type II}  on which $s>\sqrt{-\frac{1}{\theta'(t)}(t)}$ or $s<-\sqrt{-\frac{1}{\theta'(t)}(t)}$; or {\bf type III} on which $-\sqrt{-\frac{1}{\theta'(t)}(t)}<s<\sqrt{-\frac{1}{\theta'(t)}(t)}$.\\
Finally, we can further take the coordinates $(\tilde{s},\tilde{t})=(s-C, \theta(t))$ to normalize $b$ such that it only depends on $s$ and $\alpha$, then we have 
\begin{equation}\label{typfor1}
\alpha=\frac{\tilde{s}+C}{(\tilde{s}+C)^2+\frac{1}{\theta'(\theta^{-1}(\tilde{t}))}},\ a=0,\ b=\frac{\alpha}{(\tilde{s}+C)\sqrt{1+\alpha^{2}}},
\end{equation}
for some constant $C$, and hence
\begin{equation}
\zeta_{1}(\tilde{t})=C,\ \ \zeta_{2}(\tilde{t})=\frac{1}{\theta'(\theta^{-1}(\tilde{t}))}.
\end{equation}

\begin{figure}[ht]
 \begin{minipage}{0.45\textwidth}
   \centering
   \includegraphics[scale=0.45]{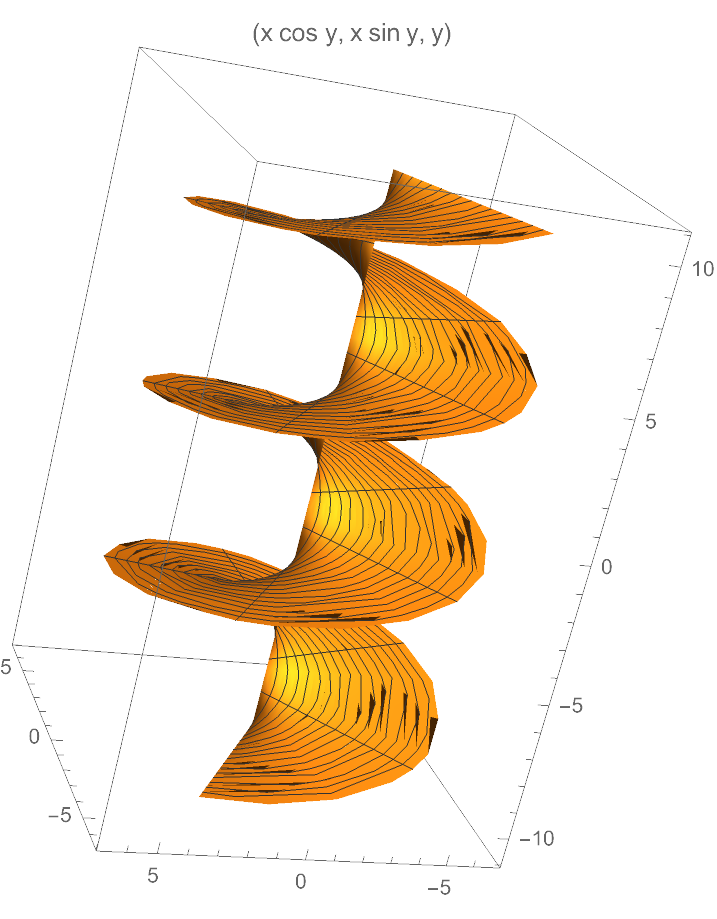}
   \caption{Helicoid}
   \label{fig5}
    \end{minipage}
 \begin{minipage}{0.45\textwidth}
   \centering
   \includegraphics[scale=0.5]{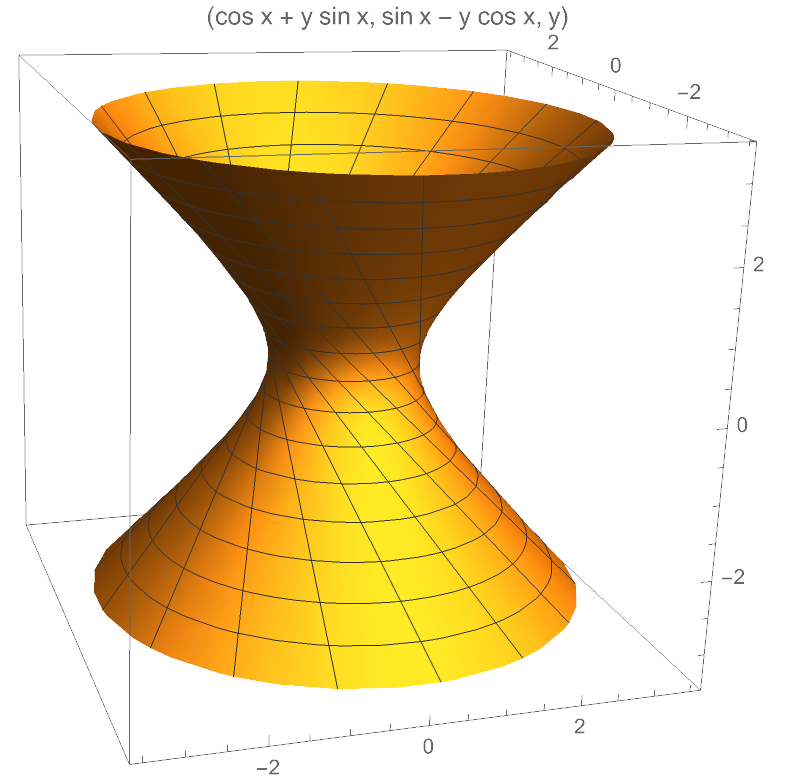}
   \caption{Conicoid}
   \label{fig6}
   \end{minipage}
\end{figure} 

\section{Structures of singular sets of $p$-minimal surfaces}\label{strofsing}

In this section, we assume that $\Sigma\subset H_{1}$ is a $p$-minimal surface.
\begin{prop}\label{typro}
Let $p$ be a singular point of a $p$-minimal surface $\Sigma$. Then there must be a characteristic line approaching this point $p$.
\end{prop}
\begin{proof}
Suppose no characteristic line approaches this point $p$. We would like to find a contradiction. Firstly, by \cite{CHMY}, there exists a small neighborhood of $p$ whose intersection with the singular set is contained in a smooth curve $\Gamma_{p}$. If the neighborhood is small enough, then, on one side of the curve $\Gamma_{p}$, we can find a compatible coordinate system $(U;x,y)$ such that  $p$ is contained on the boundary of $U$. Notice that, by our assumption, $p$ does not lie at the end of any leaf of the foliation defined by $e_{1}=\frac{\partial}{\partial x}$. Thus the image of the map defined on $U$ by $(x,y)\rightarrow(\alpha,\alpha_{x})$ is bounded on the {\bf phase plane} (see Figure \ref{dfield}). Therefore, we have that
$\lim_{(x,y)\rightarrow p}\alpha(x,y)$
is finite, which is a contradiction since $p$ is a singular point.
\end{proof}

\subsection{The proof of Theorem \ref{strofsiset}}
Due to Proposition \ref{typro}, it suffices to show this theorem for a $p$-minimal surface of some type. For general type (notice that there are no singular points for {\bf type I}), 
we choose a normal coordinate system $(x,y)$ such that $\alpha$, and $a,b$ read
\begin{equation}\label{equ51}
\alpha(x,y)=\frac{x+\zeta_{1}(y)}{(x+\zeta_{1}(y))^{2}-c^{2}(y)},\ \textrm{and}\ a=0,\ b=\frac{|\alpha|}{|x+\zeta_{1}(y)|\sqrt{1+\alpha^{2}}},
\end{equation}
where $c(y)$ is a {\bf positive} function of the variable $y$ such that $\zeta_{2}(y)=-c^{2}(y)$. Then the singular set is the graphs of the functions
\begin{equation}\label{equ52}
x=-\zeta_{1}(y)\pm c(y).
\end{equation}
By \eqref{indmet}, the induced metric $I$ (or the first fundamental form) on the regular part reads
\begin{equation}\label{equ53}
I=dx\otimes dx+\big[(x+\zeta_{1}(y))^{2}+[(x+\zeta_{1}(y))^{2}+\zeta_{2}(y)]^{2}\big]dy\otimes dy.
\end{equation}
Now we use the metric to compute the length of the singular set $\{(-\zeta_{1}(y)\pm c(y),y)\}$, where $y$ belongs to some open interval.
Let $\gamma_{\pm}(y)=(-\zeta_{1}(y)\pm c(y),y)$, which is a parametrization of the singular set. Then the square of the velocity at $y$ is
\begin{equation}\label{equ57}
|\gamma'_{\pm}(y)|^{2}=|-\zeta_{1}'(y)\pm c'(y)|^{2}+c^{2}(y)>0\ \textrm{for all}\ y,
\end{equation}
where we have used the fact that $(x+\zeta_{1}(y))^{2}=c^{2}(y)$ on the singular set. Formula (\ref{equ57}) shows that the parametrized curve $\gamma_{\pm}(y)$ of the singular set has a positive length. We omit similar proof for {\bf special type II}.

Finally, for {\bf special type I}, in terms of a normal coordinate system $(x,y)$, we have
\begin{equation}\label{equ58}
\alpha=\frac{1}{x+\zeta_{1}(y)},\ \textrm{and}\ a=0,\ b=\frac{\alpha^2}{\sqrt{1+\alpha^{2}}}.
\end{equation}
If $\gamma(y)=(-\zeta_{1}(y),y)$ is a parametrization of the singular set $\{(-\zeta_{1}(y),y)\}$ for $y$ inside some open interval, then \eqref{indmet} indicates that  the induced metric $I$ on the regular part reads
\begin{equation}\label{equ53}
dx\otimes dx+\big[(x+\zeta_{1}(y))^{2}+(x+\zeta_{1}(y))^{4}\big]dy\otimes dy
\end{equation}
and the square of the velocity at $y$ is
\begin{equation}\label{equ59}
|\gamma'(y)|^{2}=(\zeta_{1}'(y))^2+\big[(x+\zeta_{1}(y))^{2}+(x+\zeta_{1}(y))^{4}\big]=(\zeta_{1}'(y))^2,
\end{equation}
where we have used the fact that $x=-\zeta_{1}(y)$ on the singular set. From formula (\ref{equ59}), we see that the length of $\gamma(y)$ depends on whether the value $\zeta_{1}'(y)$ is zero or not, which implies that the singular set is either an isolated 
point or a smooth curve of positive length. In addition, the singular set as an isolated point happens if and only if $\zeta_{1}$ is a constant, that is, the surface is part of a plane. We thus complete the proof of this theorem \ref{strofsiset}.

\subsection{The proof of Theorem \ref{theo2}}
Around the singular point $p$, we may assume that the surface is represented by a graph $z=u(x,y)$. Let $X$ be a parametrization of the $p$-minimal surface around $p$ defined by
$X(x,y)=(x,y,u(x,y))$. Then 
\begin{equation}\label{equ512}
\begin{split}
X_{x}&=(1,0,u_{x})=\frac{\partial}{\partial x}+u_{x}\frac{\partial}{\partial t}=\mathring{e}_{1}+(u_{x}-y)\frac{\partial}{\partial t};\\
X_{y}&=(0,1,u_{y})=\frac{\partial}{\partial y}+u_{y}\frac{\partial}{\partial t}=\mathring{e}_{2}+(u_{y}+x)\frac{\partial}{\partial t},
\end{split}
\end{equation}
which yields
\[I(X_{x}, X_{x})=1+(u_{x}-y)^{2},\ \ I(X_{y}, X_{y})=1+(u_{y}+x)^{2},\ \ I(X_{x}, X_{y})=(u_{x}-y)(u_{y}+x),\]
where $I$ is the induced metric (first fundamental form) on the surface. Now we choose a horizontal normal as follows
\[e_{2}=-\frac{(u_{x}-y)\mathring{e}_{1}+(u_{y}+x)\mathring{e}_{2}}{D},\]
where $D=\left((u_{x}-y)^{2}+(u_{y}+x)^{2}\right)^{1/2}$. Then
\begin{equation}\label{fore1}
e_{1}=\frac{(u_{y}+x)\mathring{e}_{1}-(u_{x}-y)\mathring{e}_{2}}{D}
\end{equation}
is tangent to the characteristic curves. 

We first claim that either $u_{xx}(p)\neq 0$ or $(u_{xy}+1)(p)\neq 0$. Let $f(x,y)=u_{x}-y$ and let $(x(s),y(s))$ be a parametrization of the singular curve passing through $p$. Notice that we may assume, w.l.o.g., that the $x$-axis past $p$ is transverse to the singular curve, i.e., $y'\neq 0$. Since $f(x(s),y(s))=0$,
taking derivative with respect to $s$ gives $u_{xx}x'+(u_{xy}-1)y'=0$. Therefore, $(u_{xy}-1)(p)=0$ if $u_{xx}(p)=0$, and hence $(u_{xy}+1)(p)=2$.

If $u_{xx}(p)\neq 0$, we turn to compute the angle $\zeta$ between $e_{1}$ and $X_{x}$. First, from (\ref{equ512}), we have
\[I(e_{1},X_{x})=|e_{1}||X_{x}|\cos{\zeta}=(1+(u_{x}-y)^{2})^{1/2}\cos{\zeta}.\]
On the other hand, using (\ref{fore1}) to get
\[I(e_{1},X_{x})=\frac{(u_{y}+x)}{D}.\]
Combining the above two formulae, we obtain
\begin{equation}\label{equ513}
\begin{split}
\cos{\zeta}&=\frac{u_{y}+x}{D\sqrt{1+(u_{x}-y)^{2}}}=\frac{\frac{u_{y}+x}{u_{x}-y}}{\frac{D}{u_{x}-y}\sqrt{1+(u_{x}-y)^{2}}}\\
&=\pm\frac{\frac{u_{y}+x}{u_{x}-y}}{\sqrt{1+(\frac{u_{y}+x}{u_{x}-y})^{2}}\sqrt{1+(u_{x}-y)^{2}}},
\end{split}
\end{equation}
where the sign $\pm$ depends on that the sign of $u_{x}-y$ is positive or not. By the mean value theorem, it is easy to see (or see \cite{CHMY}) that
\begin{equation}\label{equ517}
\lim_{q\rightarrow p^{+}}\frac{u_{y}+x}{u_{x}-y}=\frac{u_{xy}+1}{u_{xx}}(p)=\lim_{q\rightarrow p^{-}}\frac{u_{y}+x}{u_{x}-y},
\end{equation}
and thus
\begin{equation}\label{equ514}
\lim_{q\rightarrow p^{+}}\cos{\zeta}=\frac{\frac{u_{xy}+1}{u_{xx}}(p)}{\sqrt{1+\left(\frac{u_{xy}+1}{u_{xx}}(p)\right)^{2}}}=-\lim_{q\rightarrow p^{-}}\cos{\zeta},
\end{equation}
where $\lim_{q\rightarrow p^{+}} (\lim_{q\rightarrow p^{-}})$ means that $q\rightarrow p$ from the side in which $u_{x}-y$ is positive (negative).

If $(u_{xy}+1)(p)\neq 0$, similar computations give the angle $\eta$ between $e_{1}$ and $X_{y}$ by
\begin{equation}\label{equ515}
\cos{\eta}=\frac{-(\frac{u_{x}-y}{u_{y}+x})}{\pm\sqrt{1+(\frac{u_{x}-y}{u_{y}+x})^{2}}\sqrt{1+(u_{y}+x)^{2}}},
\end{equation}
thus
\begin{equation}\label{equ516}
\lim_{q\rightarrow p^{+}}\cos{\eta}=-\frac{\frac{u_{xx}}{u_{xy}+1}(p)}{\sqrt{1+\left(\frac{u_{xx}}{u_{xy}+1}(p)\right)^{2}}}=-\lim_{q\rightarrow p^{-}}\cos{\eta},
\end{equation}
where $\lim_{q\rightarrow p^{+}} (\lim_{q\rightarrow p^{-}})$ means that $q\rightarrow p$ from the side in which $u_{y}+x$ is positive (negative).

From \eqref{equ517}, it is easy to see that  both $u_{x}-y$ and $u_{y}+x$ differ by a sign on the different side of the singular curve which is defined by $u_{x}-y=0$ and $u_{y}+x=0$. Therefore, from the formula of $e_{1}$ (see \eqref{fore1}), together with (\ref{equ514}) and (\ref{equ516}), we conclude that the characteristic vector field $e_{1}$ differs by a sign on the different side of the singular curve when approaching the singular point $p$. This completes the proof of theorem \ref{theo2}.

\subsection{The proof of Theorem \ref{main10}}
In terms of normal coordinates $(x,y)$, the surface $\Sigma$ is represented by two functions $\zeta_{1}(y)$ and $\zeta_{2}(y)$. Since it is of type II, we have $\zeta_{2}(y)<0$ and 
\begin{equation}
\alpha=\frac{x+\zeta_{1}(y)}{(x+\zeta_{1}(y))^{2}+\zeta_{2}(y)},\ a=0,\ b=\frac{|\alpha|}{|x+\zeta_{1}(y)|\sqrt{1+\alpha^{2}}},
\end{equation}
on which either $x+\zeta_{1}(y)>\sqrt{-\zeta_{2}(y)}$ or $x+\zeta_{1}(y)<-\sqrt{-\zeta_{2}(y)}$. The induced metric is 
\begin{equation}\label{641}
I=dx\otimes dx+\frac{1}{b^{2}}dy\otimes dy.
\end{equation}
We assume that $\Sigma$ lies on the part $x+\zeta_{1}(y)>\sqrt{-\zeta_{2}(y)}$ (the proof for the case that $\Sigma$ lies on the part $x+\zeta_{1}(y)<-\sqrt{-\zeta_{2}(y)}$ is similar). Suppose, in addition, that $\Sigma$ can be smoothly extended beyond the singular curve $x+\zeta_{1}(y)-\sqrt{-\zeta_{2}(y)}=0$. By theorem \ref{theo2}, the coordinates $(x,y)$ can be extended beyond the singular curve to be compatible coordinates. Then the $\alpha$-function on the other side of the singular curve must be one of the following
\begin{enumerate}
\item $\frac{1}{x+\zeta_{1}(y)-\sqrt{-\zeta_{2}(y)}}$, which is of special type I;
\item $\frac{1}{2x+2(\zeta_{1}(y)-\sqrt{-\zeta_{2}(y)})}$, which is of special type II;
\item $\alpha=\frac{x+\zeta_{1}(y)}{(x+\zeta_{1}(y))^{2}+\zeta_{2}(y)}$, which is of general type,
\end{enumerate}
for $x+\zeta_{1}(y)<\sqrt{-\zeta_{2}(y)}$. The induced metric on this other part is 
\begin{equation}\label{642}
I=dx\otimes dx-\frac{a}{b}dx\otimes dy-\frac{a}{b}dy\otimes dx+\frac{(1+a^{2})}{b^{2}}dy\otimes dy.
\end{equation}
Comparing \eqref{641} and \eqref{642}, and noting that $I$ is smooth around the singular curve, we have
\[a=0,\ b=\frac{|\alpha|}{|x+\zeta_{1}(y)|\sqrt{1+\alpha^{2}}},\]
with $\alpha=\frac{x+\zeta_{1}(y)}{(x+\zeta_{1}(y))^{2}+\zeta_{2}(y)}$. That is, cases (1) and (2) for $\alpha$ do not happen. Therefore, the extended coordinates beyond the singular curve are also normal coordinates. The formula of $\alpha$ shows that the part on the other side of the singular curve is of type III. This completes the proof of Theorem \ref{main10}.

\subsection{The proof of the Bernstein-type theorem}

In this subsection, we will show that \eqref{pmin1} and \eqref{pmin2} are the only entire smooth $p$-minimal graphs. Suppose that $\Sigma\subset H_{1}$ is an entire $p$-minimal graph. First of all, since it is a graph, we notice that there is nowhere at which $\alpha$ is zero. Next, we claim the following lemma.
\begin{lem}\label{nclogt}
The induced singular characteristic foliation of $\Sigma$ does not contain a leaf along which $\alpha$ is of general type, that is, in terms of normal coordinates around the leaf, $\alpha$ is a general solution of the {\bf Codazzi-like} equation $($see \eqref{lieq}$)$. 
 \end{lem}
 \begin{proof}
Suppose not. We assume that the induced singular characteristic foliation of $\Sigma$ contains such a leaf. Then there will be a piece of the surface (a neighborhood) around the leaf such that this piece is of general type. Suppose that  this piece is of type I or of type III, then the entireness and the phase plane (Figure \ref{dfield}) indicate that the $\alpha$-function must be extended so that it has a zero somewhere. This is a contradiction. Therefore this piece (of general type) must be of type II. Again, since it is entire, this piece can be smoothly extended through the singular curve.  By Theorem \ref{main10}, it contains a piece of type III, which lies on the other side of the singular curve. This is also a contradiction, as we argue above. We hence complete the proof of Lemma \ref{nclogt}.
 \end{proof}
 From Lemma \ref{nclogt}, we know that an entire $p$-minimal graph is either of special type I or of special type II. If it is of special type II, Theorem \ref{main09} and Lemma \ref{le02} ensure that $\Sigma$ is one of the graphs in \eqref{pmin2}. If it contains a piece of special type I, then this piece must be symmetric by Theorem \ref{main08}. Therefore, by Theorem \ref{main07}  and Lemma \ref{le01}, the surface $\Sigma$ must be one of the graphs in \eqref{pmin1}. We hence complete the proof of the Bernstein-type theorem. We also remark that the Bernstein-type theorem still holds for $C^3$ surfaces in $H_1$.
 
\begin{rem}
 We point out that in \cite{CHMY}, the Bernstein-type theorem had been proved in $C^{2}$-graphs. 
\end{rem}

 \section{An approach to construct $p$-minimal surfaces}\label{appcon}
 
In this section, we provide an approach to constructing $p$-minimal surfaces. It turns out that, in some sense, generic $p$-minimal surfaces can be constructed by this approach, particularly, other than those $p$-minimal surfaces of special type I. This approach is to perturb the surface $u=0$ in some way. Recall we choose the parametrization of $u=0$ by
\[X: (r,\theta) \rightarrow (r\cos{\theta},r\sin{\theta}, 0),\ \ r>0,\]
where each half-ray $l_{\theta}:r\rightarrow(r\cos{\theta},r\sin{\theta}, 0)$ with a fixed angle $\theta$ is a Legendrian straight line. Therefore, the image of the action of each Heisenberg rigid motion on $l_{\theta}$ is also a Legendrian straight line. Let $\mathcal{C}$ be an arbitrary curve $\mathcal{C}:\theta\rightarrow(x(\theta),y(\theta),z(\theta)),\ \theta\in\R$. Then for each fixed $\theta$ and $r>0$, the curve defined by
 \[L_{\mathcal{C}(\theta)}(l_{\theta}):r\rightarrow(x(\theta)+r\cos{\theta},y(\theta)+r\sin{\theta},z(\theta)+ry(\theta)\cos{\theta}-rx(\theta)\sin{\theta})\]
 is a Legendrian straight line. Here $L_{\mathcal{C}(\theta)}$ is the left translation by $\mathcal{C}(\theta)$. Therefore, the union of these lines constitutes a $p$-minimal surface with a parametrization $Y$ given by
 \begin{equation}
 Y(r,\theta)=(x(\theta)+r\cos{\theta},y(\theta)+r\sin{\theta},z(\theta)+ry(\theta)\cos{\theta}-rx(\theta)\sin{\theta}).
 \end{equation}
This surface depends on the curve $\mathcal{C}(\theta)=(x(\theta),y(\theta),z(\theta))$. We have the following proposition about the surface $Y$.

\begin{prop}
The coordinates $(r,\theta)$ are compatible coordinates for $Y$. In terms of this coordinate system, the $\alpha$-invariant and the induced metric read
\begin{equation}\label{705}
\begin{split}
a&=\frac{-(x'(\theta)\cos{\theta}+y'(\theta)\sin{\theta})\alpha}{[r+(y'(\theta)\cos{\theta}-x'(\theta)\sin{\theta})]\sqrt{1+\alpha^2}}\\
&\\
b&=\frac{\alpha}{[r+(y'(\theta)\cos{\theta}-x'(\theta)\sin{\theta})]\sqrt{1+\alpha^2}},
\end{split}
\end{equation}
and 
\begin{equation}\label{706}
\alpha=\frac{r+(y'(\theta)\cos{\theta}-x'(\theta)\sin{\theta})}{[r+(y'(\theta)\cos{\theta}-x'(\theta)\sin{\theta})]^{2}+\Theta(\mathcal{C}'(\theta))-(y'(\theta)\cos{\theta}-x'(\theta)\sin{\theta})^2},
\end{equation}
where $\Theta(\mathcal{C}'(\theta))=z'(\theta)+x(\theta)y'(\theta)-y(\theta)x'(\theta)$.
\end{prop}

\begin{proof}
We make a straightforward computation for the invariants $\alpha,a$ and $b$. Firstly, we have
\begin{equation}\label{701}
\begin{split}
Y_{r}&=(\cos{\theta},\sin{\theta},y(\theta)\cos{\theta}-x(\theta)\sin{\theta})\\
&=\cos{\theta}\ \mathring{e}_{1}(Y(r,\theta))+\sin{\theta}\ \mathring{e}_{2}(Y(r,\theta)).
\end{split}
\end{equation} 
From the construction of $Y$, we have $e_{1}=Y_{r}$. Thus 
\begin{equation}\label{702}
e_{2}=Je_{1}=-\sin{\theta}\ \mathring{e}_{1}(Y(r,\theta))+\cos{\theta}\ \mathring{e}_{2}(Y(r,\theta)),
\end{equation}
whereas we have
\begin{equation*}
Y_{\theta}=(x'(\theta)-r\sin{\theta},y'(\theta)+r\cos{\theta},z'(\theta)+r(y'(\theta)\cos{\theta}-y(\theta)\sin{\theta}-x'(\theta)\sin{\theta}-x(\theta)\cos{\theta})).
\end{equation*}
If we let
\begin{equation}\label{703}
Y_{\theta}=A\ \mathring{e}_{1}(Y(r,\theta))+B\ \mathring{e}_{2}(Y(r,\theta))+C\ \frac{\partial}{\partial z},
\end{equation} 
for some functions $A,B$ and $C$. Then straightforward computations show that
\begin{equation}\label{7031}
\begin{split}
A&=x'(\theta)-r\sin{\theta},\ \ B=y'(\theta)+r\cos{\theta},\\
C&=z'(\theta)-x'(\theta)y(\theta)+y'(\theta)x(\theta)+2r(y'(\theta)\cos{\theta}-x'(\theta)\sin{\theta})+r^2.
\end{split}
\end{equation}
We recall that the three invariants $\alpha,a$ and $b$ are related by
\begin{equation}\label{704}
\frac{\alpha e_{2}+T}{\sqrt{1+\alpha^2}}=aY_{r}+bY_{\theta}.
\end{equation}
If we substitute \eqref{701}, \eqref{702} and \eqref{703} into \eqref{704}, and compare the corresponding coefficients, we then obtain \eqref{705} and \eqref{706}. 
\end{proof}

\begin{rem}\label{re92}
Let $D=y'(\theta)\cos{\theta}-x'(\theta)\sin{\theta}$. By \eqref{701},\eqref{703} and \eqref{7031}, we have
\begin{equation*}
\begin{split}
0=Y_{r}\wedge Y_{\theta}&\Leftrightarrow B\cos{\theta}-A\sin{\theta}=0,\ C\cos{\theta}=0,\ C\sin{\theta}=0\\
&\Leftrightarrow r+D=0,\ C=0\\
&\Leftrightarrow r+D=0,\ \Theta(\mathcal{C}'(\theta))+2rD+r^{2}=0,\ \ \textrm{by}\ \eqref{7031},\\
&\Leftrightarrow r+D=0,\ r=-D\pm\sqrt{D^{2}-\Theta(\mathcal{C}'(\theta))}\\
&\Leftrightarrow r+D=0,\ \Theta(\mathcal{C}'(\theta))-D^{2}=0.
\end{split}
\end{equation*}
\end{rem}
We conclude that $Y$ is {\bf an immersion} if and only if either $\Theta(\mathcal{C}'(\theta))-D^{2}\neq 0\ \textrm{or}\ r+D\neq 0$ for all $\theta$, where $\Theta(\mathcal{C}'(\theta))=z'(\theta)-x'(\theta)y(\theta)+y'(\theta)x(\theta)$. 

 Formula \eqref{706} suggests the following: That $Y$ defines a $p$-minimal surface of special type depends on whether $\Theta(\mathcal{C}'(\theta))-(y'(\theta)\cos{\theta}-x'(\theta)\sin{\theta})^2$ vanishes or not.

Comparing equations \eqref{713} and \eqref{714}, it is convenient to regard surfaces of special type I as surfaces of general type with $\zeta_{2}$-invariant vanishing. Now given two arbitrary functions $\zeta_{1}$ and $\zeta_{2}$, we solve equation system  \eqref{714} for a smooth curve $\mathcal{C}(\theta)=(x(\theta),y(\theta),z(\theta))$. Since system \eqref{714} is equivalent to the following system
\begin{equation}\label{716}
\left\{\begin{split}
\zeta'_{1}(\theta)&=y''(\theta)\cos{\theta}-x''(\theta)\sin{\theta}-2\big(x'(\theta)\cos{\theta}+y'(\theta)\sin{\theta}\big),\\
\zeta_{2}(\theta)&=z'(\theta)+x(\theta)y'(\theta)-y(\theta)x'(\theta)-(y'(\theta)\cos{\theta}-x'(\theta)\sin{\theta})^2,
\end{split}\right.
\end{equation}
which is underdetermined. Therefore, the solutions always exist. For example, we can solve the first equation of \eqref{716} for $(x(\theta),y(\theta))$ and then solve for $z(\theta)$ from the second one. It turns out that we can find a smooth curve $\mathcal{C}$ such that the corresponding $p$-minimal surface $Y$ has the two given functions $\zeta_{1}$ and $\zeta_{2}$ as its $\zeta_{1}$-and $\zeta_{2}$-invariants. If $\zeta_{2}=0$, then $Y$ is of special type I.
We thus conclude, together with Theorem \ref{main09} in which states parametrizations for $p$-minimal surfaces of special type II, that we generically have provided a parametrization  
for any given $p$-minimal surface with type. In particular, we give a parametrization presentation for the eight classes of maximal $p$-minimal surfaces constructed in Subsection \ref{mami}.

Finally, we point out that these $p$-minimal surfaces constructed by curves defined by \eqref{707} and \eqref{7071} are all immersed surfaces at least (in some cases, they are embedded). This is because that $\tilde{b}\neq 0$ for all points. In particular, formula \eqref{714} says that if $\tilde{\alpha}\rightarrow 0$ then \[\tilde{b}\rightarrow\frac{1}{|\Theta(\mathcal{C}'(\theta))-(y'(\theta)\cos{\theta}-x'(\theta)\sin{\theta})^2|},\]
which is not zero. 

\begin{exa}
If we take $\mathcal{C}$ to be the curve $\mathcal{C}(\theta)=(0,0,z(\theta))$ with $z'(\theta)\neq 0$, then 
\[Y(r,\theta)=(r\cos{\theta},r\sin{\theta},z(\theta)).\]
Taking the new coordinates $(s,t)=(r,z(\theta))$, we recover the surface of general type in Subsection \ref{exagt} (see Figure \ref{fig5} for the case $z(\theta)=\theta$). 
\end{exa}

\begin{exa}
If we take $\mathcal{C}$ to be the curve $\mathcal{C}(\theta)=(-\sin{\theta},\cos{\theta},\theta)$, then 
\[Y(r,\theta)=(-\sin{\theta}+r\cos{\theta},\cos{\theta}+r\sin{\theta},r).\]
The surface of type I in Subsection \ref{exat1} (see Figure \ref{fig6}) can be recovered by taking a rotation by $\frac{\pi}{2}$ about the $z$-axis. 
\end{exa}

\bibliographystyle{plain}

\end{document}